\documentclass[review,onefignum,onetabnum]{article}

\usepackage{amsmath,amssymb,bbm,amsthm}
\usepackage{dsfont}
\usepackage{graphicx,framed}
\usepackage{url,bm}
\usepackage{color}
\usepackage{float}
\RequirePackage{algorithm}[0.1]
\usepackage{algpseudocode} 
\usepackage{subfigure}
\usepackage{epstopdf,framed}
\usepackage{mathrsfs,multirow}
\usepackage{cite}
\usepackage{comment}
\usepackage{xspace}
\usepackage[colorlinks=True]{hyperref}

\usepackage{xcolor}
\hypersetup{
  colorlinks,
  citecolor=blue,
  linkcolor=blue,
  urlcolor=blue}
  
\usepackage{cleveref}
\usepackage[left=2cm,right=2cm]{geometry}
\usepackage{authblk}

\newcommand{\rev}[1]{{\color{black}#1}} 

\renewcommand{\L}{\mathcal{B}}
\newcommand{\norm}[1]{\ensuremath{\Arrowvert #1 \Arrowvert}}
\newcommand{\inner}[2]{\langle #1, #2 \rangle}
\newcommand{\tol}{{\tt Tol }}
\newcommand{\R}{\mathds{R}}
\renewcommand{\Re}{\R}

\newcommand{\dom}[1]{\mathrm{\mathcal{D}om}(#1)} 
\newcommand{\col}[1]{\left\{#1\right\}}

\newcommand{\ind}{{\bf i}}
\newcommand{\grO}{\mathrm{O}}

\newtheorem{remark}{Remark}
\newtheorem{theorem}{Theorem}
\newtheorem{proposition}{Proposition}
\newtheorem{definition}{Definition}

\title{Computing Wasserstein Barycenter via operator splitting: the method of averaged marginals}
\date{}
\author[1,3]{D. Mimouni\thanks{daniel.mimouni@ifpen.fr}}
\author[1]{P. Malisani\thanks{paul.malisani@ifpen.fr}}
\author[2]{J. Zhu\thanks{jiamin.zhu@ifpen.fr}}
\author[3]{W. de Oliveira\thanks{welington.oliveira@minesparis.psl.eu}}
\affil[1]{\footnotesize IFP Energies nouvelles, Dpt. Applied Mathematics, 1-4 Av Bois-Préau, 92852 Rueil-Malmaison}
\affil[2]{\footnotesize IFP Energies nouvelles, Dpt. Control and Signal Processing, 1-4 Av Bois-Préau, 92852 Rueil-Malmaison}
\affil[3]{\footnotesize Mines Paris, Center for Applied Mathematics, 1, rue Claude Daunesse, F-06904 Sophia Antipolis }

\begin{document}

\maketitle

\begin{abstract}
The Wasserstein barycenter (WB) is an important tool for summarizing sets of probability measures. It finds applications in applied probability, clustering, image processing, etc. When the measures' supports are finite, computing a \rev{(balanced)} WB can be done by solving a linear optimization problem whose dimensions generally exceed standard solvers' capabilities. 
\rev{In the more general setting where measures have different total masses, we propose a convex nonsmooth optimization formulation for the so-called unbalanced WB problem. Due to their colossal dimensions, we introduce a decomposition scheme based on the Douglas-Rachford splitting method that can be applied to both balanced and unbalanced WB problem variants.}
Our algorithm, which has the interesting interpretation of being built upon averaging marginals, operates a series of simple (and exact) projections that can be parallelized and even randomized, making it suitable for large-scale datasets. Numerical comparisons against state-of-the-art methods on several data sets from the literature illustrate the method's performance.
\end{abstract}

\section{Introduction} \label{introduction}
In applied probability, stochastic optimization, and data science, a crucial aspect is the ability to compare, summarize, and reduce the dimensionality of empirical/discrete measures. Since these tasks rely heavily on pairwise comparisons of measures, it is essential to use an appropriate metric for accurate data analysis. 
Different metrics define different barycenters of a set of measures:  
a barycenter is a mean element that minimizes the (weighted) sum of all its \rev{square} distances to the set of target measures. 
When the chosen metric is the optimal transport one, and there is mass equality between the measures, the underlying barycenter is denoted by (balanced) Wasserstein Barycenter (WB).

The optimal transport metric defines the so-called Wasserstein distance (also known as Mallows or Earth Mover's distance),  a popular choice in statistics, machine learning, and stochastic optimization \cite{Peyre_Cuturi_2019,Oliveira_CEPEL_2009,Pflug_Pichler_2014}.
The Wasserstein distance has several valuable theoretical and practical properties \cite{Villani_2009,Rubner_2000} that are transferred to WBs \cite{Carlier,Cuturi_Doucet_14,Puccetti,Peyre_Cuturi_2019}.
Indeed, thanks to the Wasserstein distance, one key advantage of WBs is their ability to preserve the underlying geometry of the data, even in high-dimensional spaces. This fact makes WBs particularly useful in image processing, where datasets often contain many pixels and complex features that must be accurately represented and analyzed \cite{simon2020barycenters,tartavel2016wasserstein}. 

Being defined by the Wasserstein distance, WBs are challenging to compute. The Wasserstein distance is computationally expensive because, to compute an optimal transport plan, one needs to cope with a large linear program (LP) problem that has no analytical solution and cubic worst-case complexity\footnote{More precisely, $O(S^3log(S))$, with $S$ the size of the input data.} \cite{J.Ye}. The situation becomes even worse for computing a WB
\rev{of a set of finitely many discrete measures as the problem involves several transport plans \cite{Carlier}. This problem can be written as LP \cite{Carlier_Oberman_Oudet_2015,Anderes_Borgwardt_Miller_2016}, whose size becomes astronomical as it scales exponentially in the number of measures, exceeding thus the capabilities of standard LP solvers even for a small number of measures \cite{Carlier_Oberman_Oudet_2015,Anderes_Borgwardt_Miller_2016,Borgwardt_Patterson_2021}. 
For this reason, significant effort has been made to reduce the LP's size and design specialized solvers 
\cite{Carlier_Oberman_Oudet_2015,Anderes_Borgwardt_Miller_2016,Borgwardt_Patterson_2020,Altschuler_Boix-Adsera_2020}.
In particular, the work
\cite{Borgwardt_Patterson_2020} proposes  reduced LP models that exploit data structure. Although significantly smaller than the original LP problem defining WBs, those models are in general large scale and still hard to solve. 
The work \cite{Altschuler_Boix-Adsera_2020}
leverages techniques from computational geometry and combinatorial optimization to propose a specialized LP solver for computing WBs. 
The approach, which works on the dual problem and implements a separation oracle,
is not efficient beyond moderate-scale inputs \cite[\S\,5]{Altschuler_Boix-Adsera_2020}. Indeed, a WB cannot be computed in time polynomial in the number
of measures, (maximum) support size, and dimension \cite{Altschuler_Boix-Adsera_2022}. 
%

Given the difficulty of computing exact (free-support) WBs, much research has focused on inexact approaches. A vast body of literature focuses on computing inexact WBs, either by employing approximate LP approaches as in \cite{Puccetti,Borgwardt_2022,Lindheim_2023}, or by restricting the support of the WB to a fixed set, the so-called fixed-support approaches \cite{Cuturi_Doucet_14,J.Ye,Peyre_Cuturi_2019}. 
These techniques often employ a block-coordinate scheme consisting of two steps, first fixing the support and optimizing over the masses, then fixing the masses and optimizing over the support (of a given size).
The first of these steps is an LP problem with the same structure as the exact (free-support) WB's LP formulation discussed above. The only difference is the LP's size, as fixing the support reduces the problem significantly. The second step in the block-coordinate scheme has a straightforward solution, provided the quadratic Wasserstein distance is employed. 

Hence, whether an exact or inexact approach is employed to compute (approximate) a WB, one invariably has to face a large-scale LP of the form (see equations~\eqref{eq:Pi} and \eqref{eq:L} for details)
\begin{equation}\label{LPgen}
\min_{\pi \in \L} \; \sum_{m=1}^M \inner{c^{(m)}}{\pi^{(m)}}  \quad
         \mbox{s.t.}\quad  \pi^{(m)} \in \Pi^{(m)},\quad m=1,\ldots,M,
\end{equation}
where $M$ stands for the number of discrete measures, $c$ for the transportation costs, $\Pi^{(m)}$ represents a polytope containing measures with given marginal, and $\L$ symbolizes a linear subspace.
While exact techniques usually build upon linear programming techniques, inexact approaches tackle~\eqref{LPgen} via reformulations based  on an entropic regularization \cite{Cuturi_Doucet_14,J.Ye,Gramfort_Pyre_Cuturi_2015,Cuturi_Peyre_2016,IBP,Peyre_Cuturi_2019}.
Indeed, the work \cite{Cuturi_Doucet_14} proposes to compute a WB inexactly by decomposing \eqref{LPgen} along the measures and then regularizing the resulting optimal transportation problems with an entropy-like function.
A projected subgradient method gives rise to a minimization scheme with decomposition to deal with the high dimensions of the LP.
The regularization technique} allows one to employ the celebrated Sinkhorn algorithm \cite{Sinkhorn_1974,Cuturi_2013}, which  has a simple closed-form expression and can be implemented  efficiently using only matrix operations.
Furthermore, this technique opened the way to the \emph{Iterative Bregman Projection} (IBP) method proposed in \cite{IBP}. 
IBP is highly memory efficient for distributions with a shared support set and is considered to be one of the most effective methods to tackle fixed-support WB problems.
However, as IBP works with an approximating model and fixed support, the method falls in the class of inexact approaches.

Another approach fitting into the category of inexact methods
has been recently proposed in \cite{J.Ye}, which uses the same type of regularization as IBP but decomposes the problem into a sequence of smaller subproblems with straightforward solutions. More specifically, the approach in \cite{J.Ye} is a modification (tailored to the WB problem) of the \emph{Bregman Alternating Direction Method of Multipliers} (B-ADMM) of \cite{B-ADMM}. 
The modified B-ADMM  has been shown to compute promising results for sparse support measures and therefore is well-suited in some clustering applications. However, 
the theoretical convergence properties of the modified B-ADMM algorithm are not well understood and the approach should be considered as a heuristic.
In the same vein, the work \cite{Ye_ADMM} proposes to address the WB problem via the standard ADMM algorithm, which decomposes the problem into smaller and simpler subproblems. As mentioned by the authors in their subsequent paper \cite{J.Ye}, the numerical efficiency of the standard ADMM is still inadequate for large datasets.

\rev{To cope with the challenge of solving LPs of the form~\eqref{LPgen} resulting from computing exact (free-support) or inexact (fixed-support) WBs,} we propose a new algorithm based on the celebrated 
Douglas-Rachford splitting operator method (DR) \cite{Douglas_Rachford_1956,Eckstein_Bertsekas_1992,Accelerated_DR_2020}. Our proposal, which exploits the problem structure for decomposition, is denoted by \emph{Method of Averaged Marginals} (MAM) as at every iteration, 
the algorithm computes a barycenter approximation by 
averaging marginals issued by transportation plans that are 
updated independently, in parallel, and even randomly if necessary. 
Accordingly, the algorithm operates a series of simple and exact projections that can be carried out in parallel and even randomly. \rev{These compelling features allow for considering data sets beyond moderate sizes in the free-support setting and attaining more accurate results than entropy-based methods usually get in the fixed-support case.
Furthermore, MAM can be applied to a more general setting where measures have different total masses.}  

All the methods mentioned in the above references deal exclusively with sets of probability measures because WBs are limited to measures with equal total masses. A tentative way to circumvent this limitation is to normalize general positive measures to compute a standard (balanced) WB. However, such a naive strategy is generally unsatisfactory and limits the use of WBs in many real-life applications such as logistics, medical imaging, and others coming from the field of biology \cite{Heinemann_Klatt_Munk_2022,Sejourne_Peyre_Vialard_2023}. 
Consequently, the concept of WB has been generalized to summarize such more general measures. Different generalizations of the WB exist in the literature, and they are based on variants of \emph{unbalanced optimal transport problems} that define a distance between general non-negative, finitely supported measures by allowing for mass creation and destruction \cite{Heinemann_Klatt_Munk_2022}. Essentially, such generalizations, known as unbalanced Wasserstein barycenters (UWBs), depend on how one chooses to relax the marginal constraints. In the review paper \cite{Sejourne_Peyre_Vialard_2023} and references therein, marginal constraints are moved to the objective function with the help of divergence functions. Differently, in \cite{Heinemann_Klatt_Munk_2022} the authors replace the marginal constraints with sub-couplings and penalize their discrepancies. It is worth mentioning that UWB is more than simply copying with global variation in the measures' total masses. Generalized barycenters tend to be more robust to local mass variations, which include outliers and missing parts \cite{Sejourne_Peyre_Vialard_2023}.

For the sake of a unified algorithmic proposal for both balanced and unbalanced WBs, in this work, we consider a different formulation for dealing with sets of measures with different total masses. 
\rev{Instead of relaxing both marginal constraints in each one of the $M$ transportation plans as done in \cite{Heinemann_Klatt_Munk_2022,Sejourne_Peyre_Vialard_2023} and references therein, our formulation generalizes the balanced WB by relaxing the constraint that the barycenter is a marginal measure of all underlying transportation plans. More specifically, by using the distance function to the subspace $\L$, that is
 $\mathtt{dist}_\L(\pi):=\min_{\theta \in \L}\, \norm{\theta - \pi}$, and a penalty parameter $\gamma>0$,
we propose the following nonlinear optimization problem yielding a UWB:
\begin{equation}\label{UWBgen}
\min_{\pi } \; \sum_{m=1}^M \inner{c^{(m)}}{\pi^{(m)}} + \gamma \mathtt{dist}_\L(\pi) \quad
         \mbox{s.t.}\quad  \pi^{(m)} \in \Pi^{(m)},\quad m=1,\ldots,M.
\end{equation}
}
While our approach can be seen as an abridged alternative to the thorough methodologies of \cite{Heinemann_Klatt_Munk_2022} and \cite{Sejourne_Peyre_Vialard_2023}, its favorable structure for efficient splitting techniques combined with the good quality of the issued UWBs confirms the formulation's practical interest.

Thanks to our unified analysis,  MAM can be applied to both balanced and unbalanced WB problems without any change: all that is needed is to set up \rev{the parameter $\gamma>0$ in~\eqref{UWBgen}}.
To the best of our knowledge, MAM is the first approach capable of handling balanced and unbalanced WB problems in a single algorithm, which can be further run in a deterministic or randomized fashion.
 In addition to its versatility, MAM copes with scalability issues arising from barycenter problems, is memory efficient, and has convergence guarantees. 
As further contributions, we conduct experiments on several data sets from the literature to demonstrate the computational efficiency and accuracy of the new algorithm and make our Python codes publicly available at the link (\url{https://ifpen-gitlab.appcollaboratif.fr/detocs/mam_wb}).
 
The remainder of this work is organized as follows. \Cref{sec:background} introduces the notation
and recalls the formulation of balanced WB problems.
The proposed formulation for unbalanced WBs is presented in \Cref{sec:UWB}.
\Cref{sec:DR} briefly recalls the Douglas-Rachford splitting (DR) method and its convergence properties both in the deterministic and randomized settings.
The main contribution of this work, the Method of Averaged Marginals, is presented in \Cref{sec:MAM1}.
Convergence analysis is given in the same section by relying on the DR algorithm's properties.
\Cref{sec:num} illustrates the
numerical performance of the deterministic and randomized variants of MAM on several data sets from the literature. Numerical comparisons with \rev{the free-support method \cite{Altschuler_Boix-Adsera_2020} and fixed-support approaches in \cite{IBP} and \cite{J.Ye}} are presented for the balanced case. Then, some applications of the UWB are considered.

\section{Background on optimal transport and Wasserstein barycenter}\label{sec:background}
Throughout this work, for $\tau\geq 0$ a given scalar, the notation $\Delta_R(\tau)$ denotes the set of vectors in $\R^R_+$ adding up to $\tau$, that is,
\begin{equation}\label{eq:Delta}
\Delta_R(\tau):=\col{u \in\R^R_+:\; \sum_{i=1}^R u_i=\tau}.
\end{equation}
If $\tau=1$, then $\Delta_R(\tau)$, denoted simply by $\Delta_R$, is the $R+1$ simplex.
Let \rev{$\mathcal{P}(\Re^d)$} be the set of Borel probability measures on \rev{$\Re^d$}.
Furthermore, let $\xi$ and $\zeta$ be two random vectors having probability measures $\mu$ and $\nu$ in $\mathcal{P}(\Re^d)$, that is, $\xi\sim \mu$ and $\zeta\sim \nu$.
\rev{Their (quadratic) $2$-Wasserstein distance is given by:
\begin{equation}\label{WD}
\tag{WD}
W_2(\mu,\nu):=\left(\inf_{\pi \in U(\mu,\nu)} \int_{\R^d\times\R^d} \norm{\xi-\zeta}^2 d\pi(\xi,\zeta)\right)^{1/2},
\end{equation}
where $U(\mu,\nu)$ is the set of all probability measures on $\Re^d\times \Re^d$ having marginals $\mu$ and $\nu$. 
We denote by $W_2^2(\mu,\nu)$ the squared Wassserstein distance, i.e.,  $W_2^2(\mu,\nu) := (W_2(\mu,\nu))^2$.}

\begin{definition}[Wasserstein Barycenter - WB]
Given $M$ measures
$\{\nu^{(1)},\ldots,\nu^{(M)}\}$ in $\mathcal{P}(\Re^d)$ and $\alpha \in \Delta_M$, a \emph{Wasserstein barycenter} with weights $\alpha$ 
 is a solution to the following optimization problem
\begin{equation}\label{WB}
\min_{\mu \in \mathcal{P}(\Re^d)}\; \sum_{m=1}^M \alpha_m W_2^2(\mu,\nu^{(m)})\,.
\end{equation}
\end{definition}

\rev{
Informally, a WB $ \mu$ is a measure such that the total cost for transporting from $ \mu$ to all $\nu^{(m)}$ is minimal concerning the quadratic Wasserstein distance.} A WB $\mu$ exists in generality and, if one of the $\nu^{(m)}$ vanishes on all Borel subsets of Hausdorff dimension $d-1$, then it is also unique \cite{Carlier}. 
\rev{
In this work, we are given $M$ empirical (discrete) measures $\nu^{(m)}$ having finite support sets:
\begin{equation}\label{eq:empirical}
{\tt supp}(\nu^{(m)}):=\col{\zeta^{(m)}_1,\ldots,\zeta^{(m)}_{S^{(m)}}}
\quad \mbox{and} \quad   \nu^{(m)}=\sum_{s=1}^{S^{(m)}} q^{(m)}_s \delta_{\zeta^{(m)}},
\end{equation}
with $\delta_u$ the Dirac unit mass on $u \in \Re^d$ and $q^{(m)}\in \Delta_{S^{(m)}}$, $m=1,\ldots,M$.
In this case, the uniqueness of WB is no longer ensured in general but the following results hold \cite{Anderes_Borgwardt_Miller_2016}.
\begin{theorem}[{From \cite{Anderes_Borgwardt_Miller_2016}}] \label{Th_Borg}
    Consider $M$ empirical measures $\nu^{(m)}$ as in~\eqref{eq:empirical}.
    Then, problem~\eqref{WB} has at least one solution. 
    \begin{itemize}
        \item [a)]     Every solution $\mu$ satisfies
    \begin{equation}\label{eq:free-support}
        {\tt supp}(\mu)\subset \Xi:=\col{
        \sum_{m=1}^M \alpha_m \zeta^{(m)}_s:\; \zeta^{(m)}_s \in {\tt supp}(\nu^{(m)}),\; m=1,\ldots,M 
        }.
    \end{equation}
     \item [b)] There exists a sparse solution $\bar \mu$ such that
    \begin{equation}\label{eq:sparse}
        |{\tt supp}(\bar \mu)|\leq  T - M+1\textnormal{, where }T:=\sum_{m=1}^M S^{(m)}
    \end{equation}
    \item [c)] If $\nu^{(m)}$, $m=1,\ldots,M$, are supported on the same grid $K_1\times\cdots\times K_d$-grid 
    in $\Re^d$, and $\alpha_m=\frac{1}{M}$ for all $m$, then there exits a solution $\mu$ to~\eqref{WB} supported on  $(M(K_1-1)+1)\times\cdots\times (M(K_d-1)+1)$-grid, uniform in all directions.
    \end{itemize}
\end{theorem}
\begin{proof}
    Item a) is Proposition 1 (iii) in \cite{Anderes_Borgwardt_Miller_2016}, Items b) and c) are Theorem 2 and Corollary 1 in the same paper.
\end{proof}

Let $R:=|\Xi|$ be the number of points $\xi$ in the finite set $\Xi$. It follows from item a) that any solution $\mu$ to problem~\eqref{WB} defined with discrete measures has the form
\[
\mu= \sum_{r=1}^R  p_r \delta_{\xi_r}, \quad \mbox{with}\quad p \in \Delta_R.
\]
By letting $\mathcal{P}_\Xi(\Re^d):=\col{\mu \in \mathcal{P}(\Re^d):\, {\tt supp}(\mu) \subset \Xi}$, problem~\eqref{WB} can be reformulated as a finite-dimensional LP by replacing the constraint $\mu \in \mathcal{P}(\Re^d)$ with $\mu \in \mathcal{P}_\Xi(\Re^d)$. Indeed, by considering all the $R$ points in $\Xi$, problem~\eqref{WB} boils down to 
\begin{equation}  \label{HugeLP}
\left\{
\begin{array}{llllllllll}
\displaystyle \min_{p \in \Delta_R,\;\pi\geq 0} & \displaystyle \sum_{m=1}^M\alpha_m\sum_{r=1}^R \sum_{s=1}^{S^{(m)}} \norm{\xi_r - \zeta_s^{(m)}}^2\pi^{(m)}_{rs}\\[1.5em]
 \mbox{s.t.} & \sum_{r=1}^R \pi^{(m)}_{rs}= q^{(m)}_s,\quad \;s=1,\ldots,S^{(m)},\; m=1,\ldots,M \\[.5em]
 &\sum_{s=1}^{S^{(m)}} \pi^{(m)}_{rs} = p_r,\quad \;r=1,\ldots,R,\; m=1,\ldots,M .
\end{array}
\right.
\end{equation}
Such LP scales exponentially in the number of measures. To see that, assume that all measures have support of same cardinality $S$, i.e., $S^{(m)}=S$ for all $m=1,\ldots,M$: then $R=S^M$ and the LP has $MRS+R=M(S)^{M+1}+(S)^{M}$ variables and  $M(R+S)=M(S)^{M}+MS$ equality constraints.
When the measures are supported on the same discrete grid in $\Re^d$ and $\alpha_m=\frac{1}{M}$ for all $m$, the number of different points in $\Xi$ reduces drastically: in this case, $S=K^d$ and  $R= ((K-1)M+1)^d$ from item c) above, which is significantly smaller than $S^M = K^{dM}$ in the previous general setting (however still colossal in real-life applications) \cite{Borgwardt_Patterson_2020,Borgwardt_Patterson_2021}. These observations shed light on how the number of measures, the sizes of their support sets, and dimension $d$ impact the size of problem~\eqref{HugeLP}. The paper \cite{Borgwardt_Patterson_2021} investigates the complexity of computing a sparse Wasserstein barycenter, and \cite{Altschuler_Boix-Adsera_2022} shows that a WB cannot be computed in time polynomial in the number
of measures, (maximum) support size, and dimension $d$. 

Item b) ensures that a sparse solution exists with a support size of at most $M(S-1)+1$,
motivating thus the so-called \emph{fixed-support} approaches that generally  employ a block-coordinate optimization heuristic: 
 at iteration $k$, a support $\Xi^k$ of size $R$ (say $R\leq M(S-1)+1$) is fixed and the LP~\eqref{HugeLP} (with $\xi_r \in \Xi^k$) is solved to get an optimal plan $\pi^k$, which is in turn fixed in the optimization problem $\min_{\xi}\; \sum_{m=1}^M\alpha_m\sum_{r=1}^R \sum_{s=1}^{S^{(m)}} \norm{\xi_r - \zeta_s^{(m)}}^2\pi^{k,(m)}_{rs}$ yielding a new fixed support $\Xi^{k+1}$. Observe that this last problem has a straightforward solution (see for instance \cite[Alg. 2]{Cuturi_Doucet_14} and \cite[\S\, II]{J.Ye}).
 Otherwise, when the \emph{free-support} approach is taken, computing a WB amounts to solve~\eqref{HugeLP} by considering all points $\xi \in \Xi$, thus yielding an LP of astronomical size. 
Hence, whether an exact (free-support) or inexact (fixed-support) approach is employed to compute (approximate) a WB, one invariably has to face a large-scale LP of the form~\eqref{HugeLP}, which fits into the structure of~\eqref{LPgen} by dropping the decision variable\footnote{Although variable $p$ is the one of interest, 
it can be removed from \eqref{HugeLP} and easily recovered thanks to the balanced subspace \eqref{eq:L}.} $p$, setting
\begin{equation}\label{eq:Pi}
\Pi^{(m)}:=\left\{\pi^{(m)}\geq 0:\,\sum_{r=1}^R \pi^{(m)}_{rs}=q^{(m)}_s,\; s=1,\ldots,S^{(m)}\right\},\; m=1,\ldots,M,
\end{equation}
and the linear subspace 
\begin{equation}\label{eq:L}
\L:=\left\{\pi=(\pi^{(1)},\ldots,\pi^{(M)})\left\vert \begin{array}{lclllllll}
\sum_{s=1}^{S^{(1)}} \pi^{(1)}_{rs} &= &\sum_{s=1}^{S^{(2)}} \pi^{(2)}_{rs},& r=1,\ldots,R\\
\sum_{s=1}^{S^{(2)}} \pi^{(2)}_{rs} &= & \sum_{s=1}^{S^{(3)}} \pi^{(3)}_{rs},&r=1,\ldots,R\\
&\vdots\\
\sum_{s=1}^{S^{(M-1)}} \pi^{(M-1)}_{rs} &=&\sum_{s=1}^{S^{(M)}} \pi^{(M)}_{rs},& r=1,\ldots,R
\end{array}
\right.
\right\}.
\end{equation}
The polytope $\Pi^{(m)}$ is composed of transportation plans $\pi^{(m)}$  with right marginals $q^{(m)}$. The set with all left marginals is characterized by the linear subspace $\L$ of ``balanced plans".

In light of the above observations, we focus on a decomposition technique for solving LPs of the form~\eqref{LPgen} to render computing a (free or fixed support) WB possible beyond moderate-scale data inputs.
We mention in passing that no assumption on the costs of \eqref{LPgen} is required. This fact opens the way to consider, for instance, $W_\iota^\iota$ Wasserstein distances with $\iota \in [1,\infty)$. 

}

\section{Discrete unbalanced Wasserstein Barycenter}\label{sec:UWB} 
A well-known drawback of formulation~\eqref{WB} is its limitation to measures with equal total masses, so the feasible set defining the Wasserstein distance~\eqref{WD} is nonempty. To overcome this limitation, an idea is to relax the marginal constraints in~\eqref{WD} to cope with “unbalanced” measures, i.e., with different masses \cite{Sejourne_Peyre_Vialard_2023}. Different manners to relax these marginal constraints yield different generalizations of the concept of Wasserstein barycenter, known in the literature by the name of  \emph{unbalanced Wasserstein barycenters} (UWBs). 
In this work, we propose a new formulation that uses a metric to measure the distance of a multi-plan $\pi=(\pi^{(1)},\ldots,\pi^{(M)})$ to the balanced subspace $\L$ defined in~\cref{eq:L}. We take such a metric as being the
Euclidean distance $\mathtt{dist}_\L(\pi)=\min_{\theta \in \L}\, \norm{\theta - \pi}$
and define the following nonlinear optimization problem, with $\gamma>0$ a penalty parameter, \rev{$\Pi^{(m)}$ given in~\eqref{eq:Pi}, and $\L$ in~\eqref{eq:L}:
\begin{equation}  \label{UWB}
\left\{
\begin{array}{llllllllll}
\displaystyle \min_{\pi} & \displaystyle \sum_{m=1}^M\alpha_m\sum_{r=1}^R \sum_{s=1}^{S^{(m)}} \norm{\xi_r - \zeta_s^{(m)}}^2\pi^{(m)}_{rs}+\gamma \mathtt{dist}_\L(\pi)\\[1.5em]
 \mbox{s.t.} & \pi^{(m)} \in \Pi^{(m)},\;  m=1,\ldots,M .
\end{array}
\right.
\end{equation}
}
 This problem has always a solution because the objective function is continuous and the non-empty feasible set is compact.
Note that in the balanced case, problem~\cref{UWB} is a relaxation of~\cref{LPgen}. In the unbalanced setting,
any feasible point to~\cref{UWB} yields $\mathtt{dist}_\L(\pi)>0$. As this distance function is strictly convex outside $\L$, the above problem has a unique solution.

\begin{definition}[Discrete Unbalanced Wassertein Barycenter - UWB]\label{def:UWB}

Given a set $\{\nu^{(1)},\ldots,\nu^{(M)}\}$ of unbalanced non-negative vectors, let $\bar \pi\geq0$ be the unique solution to problem~\cref{UWB}, and $\tilde \pi$ the projection of $\bar \pi$ onto the balanced subspace $\L$, that is, $\tilde \pi:={\tt Proj}_\L(\bar \pi)$. 
The \rev{measure $\mu=\sum_{r=1}^R p_r\delta_{\xi_r}$} with $ p_r := \sum_{s=1}^{S(m)} \tilde \pi^{(m)}_{rs}$,$ r=1,\ldots,R$ (no matter $m\in\col{1,\ldots,M}$) is defined as the $\gamma$-unbalanced Wasserstein barycenter of $\{\nu^{(1)},\ldots,\nu^{(M)}\}$.
\end{definition}

The above definition differs from the ones found in the literature, which relaxes the constraints $\sum_{r=1}^R \pi_{rs}^{(m)}=q^{(m)}_s$, see for instance \cite{Sejourne_Peyre_Vialard_2023,Heinemann_Klatt_Munk_2022}. Although the above definition is not as general as the ones
of the latter references,
it 
provides meaningful results (see \Cref{sec:resul-UWB} below), uniqueness of the barycenter (if unbalanced), and is indeed  an extension of (balanced) WB \rev{as the LP \eqref{HugeLP} is for~\eqref{LPgen} what the nonlinear problem~\eqref{UWB} is for~\eqref{UWBgen}.}
\begin{proposition} \label{prop_gamma}
Suppose that
$\{\nu^{(1)},\ldots,\nu^{(M)}\}$ are probability measures and let 
$\gamma>\norm{c}$  in problem~\cref{UWBgen}.
\rev{Then $\bar \pi$ solves~\eqref{UWBgen}  if and only if $\bar \pi$ solves~\eqref{LPgen}. In particular, any UWB according to \cref{def:UWB} is also a (balanced) WB.}
\end{proposition}
\begin{proof}
Being a linear function, the objective of~\eqref{LPgen} is  Lipschitz continuous with constant $\norm{c}$. Thus, the standard theory of exact penalty methods in optimization (see for instance \cite[Prop. 1.5.2]{Bertsekas_2015}) ensures that, when
$\gamma>\norm{c}$, 
\rev{$\bar \pi$ solves problem~\cref{UWBgen} if and only if  $\bar \pi$ solves~\cref{LPgen}. In particular, for $\Pi^{(m)}$ and $\L$ given in~\eqref{eq:Pi} and \eqref{eq:L}, respectively,
 $\bar \pi={\tt Proj}_\L(\bar \pi)$ and
  the measure $\mu=\sum_{r=1}^R p_r\delta_{\xi_r}$  with  $ p_r = \sum_{s=1}^{S(m)} \bar \pi^{(m)}_{rs}$ (as in~\cref{def:UWB}) solves~\eqref{HugeLP}.}
\end{proof}

\noindent Another advantage of \cref{def:UWB} is that the problem yielding the proposed UBW enjoys a favorable structure that can be efficiently exploited by splitting methods.  \rev{Indeed, it turns out that computing a balanced or unbalanced WB can be done by the algorithm presented in \Cref{sec:MAM}. In the next section, we show that the computational burden to solve either the LP \cref{LPgen} or the nonlinear problem~\cref{UWBgen} by the Douglas-Rachford splitting method is the same.}


\section{Problem reformulation and the DR algorithm}
\label{sec:DR}

\rev{We have recalled that computing a free or fixed-support (balanced) WB requires solving one or more LPs of the form~\eqref{LPgen}, with $\Pi^{(m)}$ and $\L$ given in \eqref{eq:Pi} and~\eqref{eq:L}, respectively. Furthermore, according to our new Definition~\ref{def:UWB}, computing a UWB requires solving one (or more) nonlinear problems of the form~\eqref{UWBgen}.
In this section, we focus on problems~\cref{LPgen}
and \cref{UWBgen} and reformulate them} in a suitable way so that the Douglas-Rachford splitting operator method can be easily deployed to compute a \rev{discrete} barycenter in the balanced and unbalanced settings.
\rev{To this end, let us consider} the indicator function $\ind_C$ of a convex set $C$ (that is $\ind_C(x)=0$ if $x\in C$ and $\ind_C(x)=\infty$ otherwise) to define the convex functions
\begin{equation}  \label{eq:fm} 
    f^{(m)}(\pi^{(m)}):= 
\displaystyle \sum_{r=1}^R \sum_{s=1}^{S^{(m)}} c^{(m)}_{rs}\pi^{(m)}_{rs} +\ind_{\Pi^{(m)}}(\pi^{(m)})
,\quad m=1,\ldots,M,
\end{equation}
and recast problems~\cref{LPgen} and \cref{UWBgen}
in the following more general setting
\begin{subequations}\label{pbm}
\begin{align}
    &\displaystyle \min_{\pi }\; \displaystyle f(\pi)+g(\pi)\,, \text{   with}:\label{pbm-a}\\
    &f(\pi):=\sum_{m=1}^M f^{(m)}(\pi^{(m)})\quad \mbox{and}\quad
g(x):=\left\{
\begin{array}{ll}
     \ind_\L(\pi)  & \mbox{ if balanced} \\
    \gamma\,{\tt dist}_\L(\pi) & \mbox{ if unbalanced.}
\end{array}
\right.\label{pbm-b}
\end{align}
\end{subequations}
Since $f$ is polyhedral, \rev{$\dom{f}\cap {\tt ri}(\dom{g})\neq \emptyset$\footnote{${\tt ri}$ denotes the relative interior of a set.}}, and \cref{pbm} is solvable, \rev{it follows from \cite[Thm 27.2]{Bauschke_Combettes_2017} that}
computing one of its solutions is equivalent to 
\begin{equation}\label{GE}
\mbox{find \; $\pi $ \; such that \;  $0 \in \partial f(\pi)   + \partial g(\pi)$.}
\end{equation}
Recall that the subdifferential of a \rev{proper} convex lower semicontinuous functions is a maximal monotone operator \rev{\cite[Thm 20.40]{Bauschke_Combettes_2017}}. Thus, the above generalized equation is nothing but the problem of finding a zero of the sum of two maximal monotone operators, a well-understood problem
for which several methods exist (see, for instance, Chapters 25 and 27 of the textbook \cite{Bauschke_Combettes_2017}). Among the existing algorithms, 
the Douglas-Rachford operator
splitting method \cite{Douglas_Rachford_1956} (see also \cite[\S\, 25.2 and \S\, 27.2 ]{Bauschke_Combettes_2017}) is the most popular one.
When applied to problem~\cref{GE}, the DR algorithm asymptotically computes a solution by repeating the following steps, with $k=0,1,\ldots$, given initial point
$\theta^0 =(\theta^{(1),0},\ldots,\theta^{(M),0})$ and prox-parameter $\rho>0$:
        \begin{equation}\label{DR}
        \left\{
        \begin{array}{lll}
        \pi^{k+1} &=&\displaystyle \arg\min_{\pi }\; g(\pi)+\frac{\rho}{2} \norm{\pi -\theta^k}^2\\[1em] 
        \hat \pi^{k+1}&=& 
        \displaystyle \arg\min_{\pi }\; f(\pi)  + \frac{\rho}{2} \norm{\pi - (2\pi^{k+1}-\theta^k)}^2\\[1em]
        \theta^{k+1}&=&\theta^k + \hat \pi^{k+1}-  \pi^{k+1}.
        \end{array}
        \right.
        \end{equation}
By noting that $f$ and $g$ in~\cref{pbm-b} are \rev{proper} convex lower semicontinuous functions and problem~\cref{pbm} is solvable \rev{(so is~\eqref{GE} \cite[Thm 27.2(ii)]{Bauschke_Combettes_2017})}, the following is a direct consequence of Theorem 25.6 and Corollary 27.4 of \cite{Bauschke_Combettes_2017}.
\begin{theorem}\label{theo:DR}
    The sequence $\{\theta^k\}$ produced by the DR algorithm~\cref{DR} converges to a point $\bar \theta$, and the following holds:
    $\bar \pi:= \arg\min_{\pi }\; g(\pi)+\frac{\rho}{2} \norm{\pi -\bar \theta}^2$
        solves~\cref{pbm}, and $\{\pi^k\}$ and $\{\hat \pi^k\}$ converge to $\bar \pi$. 
\end{theorem}

The DR algorithm is attractive when the two first steps in~\cref{DR} are convenient to execute, which is the case in our settings. As we will shortly see, the iterate $\pi^{k+1}$ above has an explicit formula in both balanced and unbalanced cases, and computing $\hat \pi^{k+1}$ amounts to executing a series of independent projections onto the simplex. This task can be accomplished exactly and efficiently by specialized algorithms.

Since $f$ in~\cref{pbm-b} has a \rev{separable} structure, the computation of $\hat \pi^{k+1}$ in~\cref{DR} breaks down to a series of smaller and simpler subproblems as just mentioned. Hence, we may exploit such a structure by combining recent developments in DR's literature to produce 
the following randomized version of the DR algorithm~\cref{DR}, with $\alpha$ the vector of weights in~\cref{WB}:
\begin{equation}\label{DR-random}
        \left\{
        \begin{array}{lll}
        \pi^{k+1} &=&\displaystyle \arg\min_{\pi }\; g(\pi)+\frac{\rho}{2} \norm{\pi -\theta^k}^2\\[1em]
        &&\mbox{Draw randomly $m\in\{1,2,\ldots,M\}$ with probability $\alpha_m>0$}\\[1em]
        \hat \pi^{(m),k+1}&=& 
        \displaystyle \arg\min_{\pi^{(m)} }\; f^{(m)}(\pi^{(m)})  + \frac{\rho}{2} \norm{\pi^{(m)} - (2\pi^{(m),k+1}-\theta^{(m),k})}^2\\[1em]
        \theta^{(m'),k+1}&=&\left\{\begin{array}{ll}
             \theta^{(m),k} + \hat \pi^{(m),k+1}-  \pi^{(m),k+1} & \mbox{if $m'= m$} \\
            \theta^{(m'),k}  & \mbox{if $m'\neq m$.}
        \end{array}\right.
        \end{array}
        \right.
        \end{equation}
The randomized DR algorithm~\cref{DR-random} aims at reducing the computational burden and accelerating the optimization process. Such goals can be attained in some situations, depending on the underlying problem and available computational resources.
        The particular choice of $\alpha_m>0$ as the probability of picking up the $m^{th}$ subproblem is not necessary for convergence: the only requirement is that every subproblem is picked up with a fixed and positive probability.
The intuition behind our choice is that measures that play a more significant role in~\cref{WB} (i.e., higher $\alpha_m$) should have more chance to be picked by the randomized DR algorithm.
Furthermore, the presentation above where only one measure (subproblem) in~\cref{DR-random} is drawn is made for the sake of simplicity. One can perfectly split the set of measures into $nb < M$ bundles, each containing a subset of measures, and select randomly bundles instead of individual measures. Such an approach proves advantageous in a parallel computing environment with $nb$ available machines/processors (see \cref{IBP_vs_MAM} in the numerical section).
The almost surely (i.e., with probability one) convergence of the randomized DR algorithm depicted in~\cref{DR-random} can be summarized as follows \cite[Thm 2]{Iutzeler_2013}.
\begin{theorem}\label{theo:DR-random}
    The sequence $\{\pi^k\}$ produced by the randomized DR algorithm~\cref{DR-random} 
    \rev{converges almost surely to
a random variable  $\bar \pi$ taking values in the solution set of problem~\cref{pbm}.}
\end{theorem}
 \rev{
 This result is a special case of a thorough analysis given in \cite{Combettes_Pesquet_2015} (see, in particular,  Remark 3.5 and Section 5 in that paper.) We note that the practical performance of the randomized scheme \eqref{DR-random} depends on computational resources and is thus not always effective (see Figure~\ref{IBP_vs_MAM} below). The deterministic and asynchronous decomposition methods in \cite{Combettes_Ecstein_2018} provide significantly more flexibility in selecting the measure $\nu^{(m)}$  (or even part of it) activated at every iteration and thus may perform better than the randomized scheme above.
 As these methods do not follow the general lines of the DR algorithm, we leave the specialization of such approaches to the WB problem for future research.
 }

In the next section, we further exploit the structure of functions $f$ and $g$ in~\cref{pbm} and rearrange terms in the schemes~\cref{DR} and~\cref{DR-random} to provide an easy-to-implement and memory-efficient algorithm for computing balanced and unbalanced WBs.

\section{The Method of Averaged Marginals}~\label{sec:MAM1}
Both deterministic and randomized DR algorithms above require evaluating the proximal mapping of the function $g$ given in~\cref{pbm-b}. 
In the balanced WB setting, $g$ is the indicator function of $\L$ given in \cref{eq:L}, and thus $\pi^{k+1}$ in~\eqref{DR} is the projection of $\theta^k$ onto $\L$: $\pi^{k+1}= {\tt Proj}_\L(\theta^k)$. On the other hand, in the unbalanced WB case, $g(\cdot)$ is the penalized distance function $\gamma\, {\tt dist}_\L(\cdot)$. Computing $\pi^{k+1}$ then amounts to evaluating the proximal mapping of the distance function: $ \min_{\pi }\;{\tt dist}_\L(\pi)+\frac{\rho}{2\gamma} \norm{\pi -\theta^k}^2$. The unique solution to this problem is given by \rev{\cite[Example 24.28]{Bauschke_Combettes_2017}}
\begin{equation}
\label{eq:Unbalanced-prox}
\pi^{k+1} = 
\left\{\begin{array}{llll}
   \mathtt{Proj}_{\L}(\theta^k)  & \mbox{ if }\; \rho\,\mathtt{dist}_\L(\theta^k) \leq \gamma \\ 
  \theta^k +\frac{\gamma}{\rho\,\mathtt{dist}_\L(\theta^k)}(\mathtt{Proj}_{\L}(\theta^k)-\theta^k)   & \mbox{ otherwise.} 
\end{array}
\right.
\end{equation}
Hence, computing $\pi^{k+1}$ in both balanced and unbalanced settings boils down to projecting onto the balanced subspace. This fact allows us to provide a unified algorithm for WB and UWB problems.

\subsection{Projecting onto the subspace of balanced plans}\label{sec:proj}
In what follows we exploit the particular geometry of $\L$ to provide an explicit formula for projecting onto this linear subspace.

\begin{proposition} \label{prop_2} 
With the notation of \Cref{sec:background}, let $\theta \in \Re^{R\times T}$,
\begin{subequations}
\begin{equation}\label{pr}
a_m:=\frac{\frac{1}{S^{(m)}}}{\sum_{j=1}^M\frac{1}{S^{(j)}}},\quad 
p^{(m)}:=\left(\sum_{s=1}^{S^{(m)}}\theta_{rs}^{(m)}\right)_{1\leq r\leq R},
\quad \mbox{and}\quad p:=\sum_{m=1}^M a_mp^{(m)}.
\end{equation}
The  projection $\pi={\tt Proj}_{\L}(\theta)$ has the explicit form:
\begin{equation}\label{row_sol}
\pi^{(m)}_{rs} := \theta^{(m)}_{rs} +\frac{(p_r-p^{(m)}_r)}{S^{(m)}},\quad s=1,\ldots,S^{(m)},\;r=1,\ldots,R,\; m=1,\ldots,M.
\end{equation}
\end{subequations}
\end{proposition}
\begin{proof}
First, observe that $\pi={\tt Proj}_{\L}(\theta)$ solves the QP problem
\begin{equation}
\left\{
\begin{array}{llllllll}
\displaystyle \min_{y^{(1)},\ldots,y^{(M)}} & \displaystyle \frac{1}{2} \sum_{m=1}^M \norm{y^{(m)}- \theta^{(m),k}}^2\\
\mbox{s.t} &  \sum_{s=1}^{S^{(m)}} y^{(m)}_{rs} = \sum_{s=1}^{S^{(m+1)}} y^{(m+1)}_{rs},&\; r=1,\ldots,R,\, m=1,\ldots,M-1,\\
\end{array}
\right.
\end{equation}
which is only coupled 
by the ``columns" of $\pi$: there is no constraint linking $\pi^{(m)}_{rs}$ with $\pi^{(m')}_{r's}$ for $r\neq r'$ and $m$ and $m'$ arbitrary. Therefore, we can decompose it by rows: for $r=1,\ldots,R$, the $r^{th}$-row $(\pi^{(1)}_{r1},\ldots,\pi^{(1)}_{rS^{(1)}},\ldots,\pi^{(M)}_{r1},\ldots,\pi^{(M)}_{rS^{(M)}})$ of $\pi$ is the unique solution to the problem
\begin{equation}\label{eq:proj_row}
\left\{
\begin{array}{lllllll}
\displaystyle \min_{w} & \displaystyle \frac{1}{2} \sum_{m=1}^M
\sum_{s=1}^{S^{(m)}} \Big(w^{(m)}_{s} - \theta^{(m)}_{rs}\Big)^2\\
\mbox{s.t} &  \sum_{s=1}^{S^{(m)}} w^{(m)}_{s} = \sum_{s=1}^{S^{(m+1)}} w^{(m+1)}_{s},\; m=1,\ldots,M-1. 
\end{array}
\right.
\end{equation}
The Lagrangian function to this problem is, for a dual variable $u$, given by
\begin{equation}
L_r(w,u)= \displaystyle  \frac{1}{2} \sum_{m=1}^M
\sum_{s=1}^{S^{(m)}} \Big(w^{(m)}_{s} - \theta^{(m)}_{rs}\Big)^2
+  \sum_{m=1}^{M-1} u^{(m)}\Big( \sum_{s=1}^{S^{(m)}} w^{(m)}_{s} -  \sum_{s=1}^{S^{(m+1)}} w^{(m+1)}_{s} \Big).
\end{equation}

\noindent A primal-dual $(w,u)$ solution to problem~\cref{eq:proj_row} must satisfy the Lagrange system, in particular $\nabla_{w} L_r(w,u)=0$ with $w$ the $r^{th}$ row of $\pi={\tt Proj}_{\L}(\theta)$, that is,
\begin{equation}\label{eq:sytemL}
\left\{
\begin{array}{lclll}
    \pi^{(1)}_{rs}-\theta^{(1)}_{rs} & +& u^{(1)}&=0 & s=1,\ldots, S^{(1)}\\
    \pi^{(2)}_{rs}-\theta^{(2)}_{rs} & +& u^{(2)}-u^{(1)}&=0 & s=1,\ldots, S^{(2)}\\
    &\vdots& \\
\pi^{(M-1)}_{rs}-\theta^{(M-1)}_{rs} & + &u^{(M-1)}-u^{(M-2)}&=0 & s=1,\ldots, S^{(M-1)}
\\
\pi^{(M)}_{rs}-\theta^{(M)}_{rs} & -& u^{(M-1)}&=0 & s=1,\ldots, S^{(M)}.
\end{array}
\right.
\end{equation}
Let us denote $p_r= \sum_{s=1}^{S^{(m)}} \pi^{(m)}_{rs}$ (no matter $m\in \{1,\ldots,M\}$ because $\pi \in \L$), $p_r^{(m)}=\sum_{s=1}^{S^{(m)}} \theta^{(m)}_{rs}$ (the $r^{th}$ component of $p^{(m)}$ as defined in~\cref{pr}), 
and sum over $s$ the first row of system~\cref{eq:sytemL}  to get 
\begin{equation}
p_r- p_r^{(1)}+ u^{(1)}S^{(1)} = 0\quad \Rightarrow \quad  u^{(1)}=
\frac{ p_r^{(1)}-p_r}{S^{(1)}},
\end{equation}

\noindent Now, by summing the second row in \cref{eq:sytemL} over $s$  we get
\begin{equation}
p_r- p_r^{(2)}
 + u^{(2)}{S^{(2)}}-u^{(1)}{S^{(2)}} = 0\quad \Rightarrow \quad u^{(2)}
 =u^{(1)}+
 \frac{ p_r^{(2)}-p_r}{S^{(2)}}.
\end{equation}
By proceeding in this way and setting $u^{(0)}:=0$  we obtain
\begin{subequations}
\begin{align}\label{u_sol}
  u^{(m)}
  &= u^{(m-1)}+ \frac{ p_r^{(m)} -p_r}{S^{(m)}}
  ,\quad m=1,\ldots,M-1.
\end{align}
Furthermore, for $M-1$ we get the alternative formula
$u^{(M-1)} = -\frac{ p_r^{(M)} -p_r}{S^{(M)}}$.
\end{subequations}
Given these dual values, we can use \cref{eq:sytemL} to conclude that the $r^{th}$ row of $\pi ={\tt Proj}_\L(\theta)$ is given as in~\cref{row_sol}.
%
It is remaining to show that $p_r= \sum_{s=1}^{S^{(m)}} \pi^{(m)}_{rs}$, as defined above, is alternatively given by~\cref{pr}.
To this end, observe that $u^{(M-1)} =u^{(M-1)}-u^{(0)}=\sum_{m=1}^{M-1}(u^{(m)}-u^{(m-1)})$, so:
    

\begin{equation} \label{app-aux0}
u^{(M-1)} = \sum_{m=1}^{M-1} \Big( \frac{ p_r^{(m)} -p_r}{S^{(m)}}\Big) =\sum_{m=1}^{M-1}  \frac{ p_r^{(m)} }{S^{(m)}}
-p_r\sum_{m=1}^{M-1}  \frac{1}{S^{(m)}}.
\end{equation}

\noindent Recall that $u^{(M-1)} = \frac{p_r- p_r^{(M)} }{S^{(M)}}$, i.e.,
$ 
p_r= p_r^{(M)} + u^{(M-1)}S^{(M)}.
$ 
Replacing $u^{(M-1)}$ with the expression~\cref{app-aux0} yields
\begin{equation}
p_r  = S^{(M)}\Big[\frac{p_r^{(M)}}{S^{(M)}} + u^{(M-1)}\Big] 
 = S^{(M)}\Big[\frac{p_r^{(M)}}{S^{(M)}} + \sum_{m=1}^{M-1}  \frac{ p_r^{(m)} }{S^{(m)}}
-p_r\sum_{m=1}^{M-1}  \frac{1}{S^{(m)}}\Big],
\end{equation}
which implies $p_r \sum_{m=1}^{M}  \frac{1}{S^{(m)}} = \sum_{m=1}^{M} \Big( \frac{ p_r^{(m)} }{S^{(m)}}\Big)$.
Hence, $p$ is as given in~\cref{pr}, and the proof is complete. 
\end{proof}

Note that projection can be computed in parallel over the rows, and the average $p$ of the marginals $p^{(m)}$ is the gathering step between parallel processors.
 

\subsection{Evaluating the Proximal Mapping of Transportation Costs} \label{sec:proj_simplex}
In this subsection we turn our attention to the DR algorithm's second step, which requires solving a convex optimization problem of the form: $\min_{\pi }\; f(\pi)+\frac{\rho}{2} \norm{\pi -y}^2$ (see \cref{DR}).
Given the additive structure of $f$ in~\cref{pbm-b}, the above problem can be decomposed into $M$ smaller ones
\begin{equation}\label{eq:prox-s}
\min_{\pi^{(m)} }\; f^{(m)}(\pi^{(m)})+\frac{\rho}{2} \norm{\pi^{(m)} -y^{(m)}}^2, \quad  m=1,\ldots,M.
\end{equation}
Then looking closely at every subproblem above, we can see that we can decompose it even more: 
the columns of the the transportation plan $\pi^{(m)}$ are independent in the minimization. 
Besides, as the following result shows, every column optimization is simply the projection of an $R$-dimensional vector onto the simplex $\Delta_R$.

\begin{proposition} \label{prop_3}
Let $\Delta_R(\tau)$ be as in~\cref{eq:Delta}.     
The minimization $\hat \pi:=\min_{\pi }\; f(\pi)+\frac{\rho}{2} \norm{\pi -y}^2$
can be performed exactly, in parallel along the columns of each transport plan $y^{(m)}$, as follows: 
for  all $m \in \{1,\ldots,M\}$,
\begin{equation}\label{eq:proj-s}
    \begin{pmatrix}\hat \pi^{(m)}_{1s}\\ \vdots \\ \hat \pi^{(m)}_{Rs}\end{pmatrix}=\,{\tt Proj}_{\Delta_R(q_s^{(m)})}
    \begin{pmatrix}
     y_{1s} - \frac{1}{\rho}c^{(m)}_{1s}\\
     \vdots  \\ 
     y_{Rs} - \frac{1}{\rho}c^{(m)}_{Rs} \end{pmatrix},
    \quad s=1,\ldots, S^{(m)}.
\end{equation}
\end{proposition}
\begin{proof}
 It has already been argued that evaluating this proximal mapping  into $M$ smaller subproblems \cref{eq:prox-s}, which is a quadratic program problem due to the definition of $f^{(m)}$ in~\cref{eq:fm}:
 \begin{equation}
\displaystyle \min_{\pi^{(m)}\geq 0} \; \sum_{r=1}^R \sum_{s=1}^{S^{(m)}} \Big[c^{(m)}_{rs}\pi^{(m)}_{rs} + \frac{\rho}{2} \left|\pi^{(m)}_{rs}-y^{(m)}_{rs})\right|^2\Big]\quad 
\mbox{s.t.} \quad \sum_{r=1}^R \pi^{(m)}_{rs}=q^{(m)}_s,\; s=1, \ldots,S^{(m)}.
\end{equation}
By taking a close look at the above problem, we can see that the objective function is decomposable, and the constraints couple only the ``rows" of $\pi^{(m)}$. Therefore, we can go further and decompose the above problem per columns: for $s=1,\ldots,S^{(m)}$, the $s^{th}$-column of $\hat \pi^{(m)}$ 
is the unique solution to the $R$-dimensional problem
\begin{equation}
\displaystyle \min_{w\geq 0} \;
\displaystyle \sum_{r=1}^R \Big[c^{(m)}_{rs}w_{r} + \frac{\rho}{2} (w_{r}-y^{(m)}_{rs})^2\Big] \quad
\mbox{s.t.} \quad \sum_{r=1}^R w_r=q^{(m)}_s,
\end{equation}
which is nothing but \cref{eq:proj-s}. Such projection can be performed exactly~\cite{Condat_2016}.
\end{proof}
\begin{remark}\label{rem:ProjSimplex}
If $\tau=0$, then $\Delta_R(\tau)=\{0\}$ and the projection onto this set is trivial. 
Otherwise, $\tau >0$ and computing $\mathtt{Proj}_{\Delta_R(\tau)}(w)$ amounts to projecting onto the
$R+1$ simplex $\Delta_R$: $\mathtt{Proj}_{\Delta_R(\tau)}(w) \quad = \quad \tau\,\mathtt{Proj}_{\Delta_R}(w/\tau)$.
The latter task can be performed \emph{exactly} by using efficient methods \cite{Condat_2016}.
Hence, evaluating the proximal mapping in \cref{prop_3} decomposes into $T$ independent projections onto $\Delta_R$.
\end{remark}

\if{ 
\noindent The proof is given in details in the \cref{proof_step_2}.

\noindent It is worth noting that the projection onto the simplex is a nonlinear operation and can be computationally expensive for large-scale problems, but we deal with a very parallelisable problem, this enable us to remain effective. We would refer to L. Condat [REFERENCE ALGO EN C JE METS L'URL ? OU EN BIBLIO ? https://lcondat.github.io/software.html, l'article ? Fast projection onto the simplex and the ll1
 ball OU LES 2 ?], who developped an exact algorithm to compute the pojection in a very competitive complexity.

\subsection{Free support Wasserstein barycenter generalization} \label{subsec:free-WB}
In a context of a more general Wasserstein barycenter problem where the support of $\xi$ is not fixed, the optimization problem can be written as :
\begin{equation}\label{eq:WB2}
\min_{p \in \Delta_R, \xi}\;\sum_{m=1}^M \alpha_m{\tt OT}(p,q^{(m)};{\xi,\zeta^{(m)}}),
\end{equation}
where ${\tt OT}(\cdot;\cdot)$ is given in~\cref{OT}.
Within this framework, the scenarios $\{\xi_1,\ldots,\xi_R\}$ for $\xi$ are decision variables. To handle this more general problem, we will assume that the metric function $d$ is given by $d(\xi_r,\zeta_s)= \frac{1}{2}\|\xi_r-\zeta_s\|^2$ and use the standard technique (in the clustering and scenario reduction literature) of minimizing the objective in~\cref{eq:WB2} in two iterative steps repeated until no further improvement is obtained: first solve on $p$ with fixed support using the previous section theory, then solve on $\xi$ with fixed probabilities (say \textit{weights} in the unbalanced case). The latter step can be derived in an explicit way.

\noindent \textbf{Proposition 4.} 
The unique solution of the scenario/support optimization problem, written as
\begin{equation}\label{opti_support}
    \xi =\arg\min_{\xi }\sum_{m=1}^M \alpha_m{\tt OT}(p,q^{(m)};{\xi,\zeta^{(m)}})
\end{equation}
is given by:
\[
\xi_r :=  \frac{\sum_{m=1}^M \alpha_m  \sum_{s=1}^{S^{(m)}} \zeta^{(m)}_s\pi_{rs}^{(m)}}{ \sum_{m=1}^M \alpha_m  \sum_{s=1}^{S^{(m)}}\pi_{rs}^{(m)}}, \quad r=1,\ldots,R.
\]

\noindent \textit{Proof.} 
Let $\pi$ be the transportation plan associated to the weights $p$ (\textit{probabilities} in the balanced case).
Observe that problem~\cref{opti_support} can be written as
\[
\min_{\xi }\; \Psi(\xi), \quad \mbox{with}\quad \Psi(\xi):=\sum_{m=1}^M \alpha_m \sum_{r=1}^R \sum_{s=1}^{S^{(m)}}\frac{1}{2}\|\xi_r-\zeta^{(m)}_s\|^2\pi_{rs}^{(m)}.
\]
The function $\Psi(\cdot)$ is strictly convex and coercive. Hence, the above problem has a unique solution $\xi$, which can be computed by solving the system $\nabla \Psi(\xi)=0$.
To this end, note that for $r=1,\ldots,R$,
\[
\frac{\partial \Psi}{\partial \xi_r}(\xi)=  \sum_{m=1}^M \alpha_m  \sum_{s=1}^{S^{(m)}} (\xi_r-\zeta^{(m)}_s)\pi_{rs}^{(m)}.
\]
Therefore, $\frac{\partial \Psi}{\partial \xi_r}(\xi)=0$ implies the stated formula.\qed

}\fi

\subsection{The Method of Averaged Marginals (MAM)} \label{sec:MAM}


\rev{Our approach is presented in Algorithm~\ref{alg}, which gathers the three main steps from the DR algorithm and integrates a choice of $\gamma$ for a simple switch between the balanced and unbalanced cases.}

\begin{algorithm}[htb]
\caption{\sc Method of Averaged Marginals - MAM }
  \label{alg}
\begin{algorithmic}[1]

\Statex \Comment{Step 0: input}
\State Given $\rho>0$, the \rev{cost matrix} and initial point $c,\theta^{0} \in \Re^{R\times T}$, and $a \in \Delta_M$ as in~\cref{pr},  set $k\gets 0$ and $p_r^{(m)} \gets  \sum_{s=1}^{S^{(m)}}\theta_{rs}^{(m),0}$, $r=1,\ldots,R$,  $m=1,\ldots,M$
\State Set $\gamma\gets \infty$ if $q^{(m)}\in  \R^{S^{(m)}}_+$, $m=1,\ldots,M$, are balanced; otherwise, choose $\gamma \in [0,\infty)$ 
%
\Statex  
\While{not converged}
\Statex \Comment{Step 1: average the marginals}
\State Compute  $p^k\gets \sum_{m=1}^M a_mp^{(m)}$
\Statex  
\State Set $t^k=1$ if ${\rho\,\sqrt{\sum_{m=1}^M \frac{\norm{p^k- p^{(m)}}^2}{S^{(m)}}}}\leq \gamma$; otherwise, $t^k\gets {\gamma}/\left({\rho\,\sqrt{\sum_{m=1}^M \frac{\norm{p^k- p^{(m)}}^2}{S^{(m)}}}} \right)$ \label{line:tk}

\Statex 
\State Choose an index set $\emptyset\neq  \mathcal{M}^k \subseteq \{1,\ldots,M\}$ \label{line:random}
\For{$m \in \mathcal{M}^k$} \label{line:loop-m}
\Statex\Comment{Step 2: update the $m^{th}$ plan}
\For{$s=1,\ldots,S^{(m)}$} 
\State Define $w_r \gets  \theta^{(m),k}_{rs} +2\,t^k\,\frac{p_r^k-p_r^{(m)}}{S^{(m)}} - \frac{1}{\rho}c^{(m)}_{rs}$,  $r=1,\ldots, R$\label{line:w}

\State Compute $(\hat \pi^{(m)}_{1s},\ldots,\hat \pi^{(m)}_{Rs}) \gets {\tt Proj}_{\Delta_R( q^{(m)}_s)}(w)$\label{line:proj}

\State Update   
$\theta^{(m),k+1}_{rs} \gets \hat \pi^{(m)}_{rs} - t^k\frac{p_r^k-p_r^{(m)}}{S^{(m)}}$,  $r=1,\ldots, R$\label{line:k+1}

\EndFor
\Statex \Comment{Step 3: update the $m^{th}$ marginal}
\State Update $p^{(m)}_r \gets  \sum_{s=1}^{S^{(m)}}\theta_{rs}^{(m),k+1}$, $r=1,\ldots,R$
\EndFor
\Statex
\EndWhile

\State Return $\bar p \gets p^k$ 

  \end{algorithmic}
\end{algorithm}

\paragraph{MAM's interpretation}
At every iteration, the barycenter approximation $p^k$ is a weighted average of the $M$  marginals $p^{(m)}$ of the plans $\theta^{(m),k}$, $m=1,\ldots,M$. As we will shortly see, the whole sequence $\{ p^k \}$ converges (almost surely or deterministically) to a barycenter upon specific assumptions on the choice of the index set at line~\ref{line:random} of \cref{alg}.

\paragraph{Initialization}
The choices for $\theta^0 \in \Re^{R\times T}$ and $\rho>0$ are arbitrary ones. 
The prox-parameter $\rho>0$ is borrowed from the DR algorithm, which is known to have an impact on the 
practical convergence speed. Therefore, $\rho$
should be tuned for the set of distributions at stakes. 
Some heuristics for tuning this parameter exist for other methods derived from the DR algorithms \cite{Watson_Woodruff_2010,ADMM_Figueiredo_2017} and can be adapted to the setting of \cref{alg}. 


\paragraph{Stopping criteria}
A possible stopping test is  $\norm{\theta^{k+1}-\theta^k}_\infty \leq \tol $, where $\tol>0$ is a given tolerance. 
Alternatively, we may stop the algorithm when $\norm{p^{k+1}-p^k}\leq \tol$. \rev{or $\textnormal{dist}_B(\hat \pi^{k}) \leq \tol$. These latter tests should be understood as heuristic criteria.}

\paragraph{Deterministic and random variants of MAM}
The most computationally expensive step of MAM is Step 2, which requires a series of independent projections onto the $R+1$ simplex (see \cref{rem:ProjSimplex}).
Our approach underlines that this step can be 
conducted in parallel over $s$ or, if preferable, over the 
measures $m$. As a result, it is a natural idea to derive a 
randomized variant of the algorithm. This is the reason for 
having the possibility of choosing an index set $\mathcal{M}^k\subsetneq \{1,\ldots,M\} $ at 
line~\ref{line:random} of \cref{alg}.
For example, we may employ an economical rule and choose $\mathcal{M}^k=\{m\}$ randomly (with a fixed and positive probability, e.g. $\alpha_m$) at every iteration, or the costly one  $\mathcal{M}^k=\{1,\ldots,M\}$ for all $k$. The latter yields the deterministic method of averaged marginals, while the former gives rise to a randomized variant of MAM.
Depending on the computational resources, intermediate choices between these two extremes can perform better in practice. 
\begin{remark}\label{rem:random}
Suppose that $1<{\tt nb}< M$ processors are available. We may then create a partition $A_1,\ldots, A_{\tt nb}$ of the set $\{1,\ldots,M\}$ ($=\cup_{i=1}^{\tt nb}A_i$) and define weights $\beta_i:= \sum_{m \in A_i} \alpha_m >0$. Then, at every iteration $k$, we may draw with probability $\beta_i$ the subset $A_i$ of measures and set $\mathcal{M}^k= A_i$.
\end{remark}
This randomized variant would enable the algorithm to compute more iterations per time unit but with less precision per iteration (since not all the marginals $p^{(m)}$ are updated). Such a randomized variant of MAM is benchmarked against its deterministic counterpart in \Cref{quantitative_section}, where we demonstrate empirically that with certain configurations (depending on the number $M$ of probability distributions and the number of processors) this randomized algorithm can be effective.
We highlight that other choices for $\mathcal{M}^k$ rather than randomized ones or the deterministic rule $\mathcal{M}^k=\{1,\ldots,M\}$ should be understood as heuristics. Within such a framework, one may choose $\mathcal{M}^k \subsetneq \{1,\ldots,M\}$ deterministically, for instance cyclically or yet by the discrepancy of the marginal $p^{(m)}$ with respect to the average $p^k$.

\paragraph{Storage complexity} \label{storage}
Note that the operation at line~\ref{line:proj} is trivial if $q_s^{(m)}=0$. This motivates us to remove all the zero components of $q^{(m)}$ from the problem's data, and consequently, all the columns $s$ of the cost matrix $c^{(m)}$ and variables $\theta,\hat \pi$ corresponding to $q^{(m)}_s=0$, $m=1,\ldots,M$. In some applications (e.g. general sparse problems), this strategy significantly reduces the WB problem and thus memory allocation, since the non-taken columns are both not stored and not treated in the \textit{for loops}. This remark raises the question of how sparse data impacts the practical performance of MAM.  \Cref{parametric_section} conducts an empirical analysis on this matter. 

In nominal use, the algorithm needs to store the decision variables $\theta^{(m)} \in \mathbb{R}^{R \times S^{(m)}}$ for all $m=1,\ldots,M$ (transport plans for every measure), along with $M$ distance matrices $c \in \mathbb{R}^{R \times S^{(m)}}$, one barycenter approximation $p^k \in \mathbb{R}^R$, $M$ approximated marginals $p^{(m)}\in \mathbb{R}^R$ and $M$ marginals $q^{(m)} \in \R^{S(m)}$. Note that in practical terms, the auxiliary variables $w$ and $\hat \pi$ in \cref{alg} can be easily removed from the algorithm's implementation by merging lines~\ref{line:w}-\ref{line:k+1} into a single one. 
Hence, by letting $T=\sum_{m=1}^MS^{(m)}$, the method's memory allocation is $2R\,T+T+M(R+1)$  floating-points. This number can be reduced if the measures share the same cost matrix, i.e., $c^{(m)}=c^{(m')}$ for all $m,m'=1,\ldots,M$. In this case, $S^{(m)}=S$ for all $m$, $T=M\,S$ and the method's memory allocation drops to $R\,T+R\,S+T+M(R+1)$  floating-points. In the light of the previous remark this memory complexity should be treated as an upper bound: the sparser the data the less memory will be needed.

\rev{
\paragraph{Computation complexity}
Step 2 of the algorithm involves two main components: projection onto $\L$, which comprises straightforward operations detailed in \Cref{sec:proj}, and projection onto $\Pi$, which relies on leveraging the simplex projection technique discussed in \Cref{sec:proj_simplex}. For each probability measure, the projection onto $\L$ requires  $3R S^{(m)}$ computation operations, where $R$ is the barycenter support size and $S^{(m)}$ denotes the support size of the probability measure $m$, which undergoes iteration over its columns. On the other hand, the simplex projection (line~\ref{line:proj}) is computationally more intensive. This is due to the adoption of a state-of-the-art algorithm proposed by Condat \cite{Condat_2016}, which operates in $\grO(R\log(R))$. Therefore, the complexity of $S^{(m)}$ times line~\ref{line:proj} amounts to $\grO(S^{(m)} R\log(R))$. Deriving the precise number of computational operations is challenging due to the use of a sorting algorithm in \cite{Condat_2016}, the complexity of which depends on the characteristics of the input data, hence the asymptotic complexity estimation. Note that Step 2 must be executed for the $M$ probability measures, but these operations can be performed in parallel (multiprocessing).

}

\paragraph{Balanced and unbalanced settings}
As already mentioned, our approach can handle both balanced and unbalanced WB problems. All that is necessary is to choose a finite (positive) value for the parameter $\gamma$ in the unbalanced case. Such a parameter is only used to define $t^k \in (0,1]$ at every iteration. Indeed, \cref{alg} defines $t^k=1$ for all iterations if the WB problem is balanced (because $\gamma=\infty$ in this case)\footnote{Observe that line~\ref{line:tk} can be entirely disregarded in this case, by setting $t^k=t=1$ fixed at initialization.}, and $t^k={\gamma}/\left({\rho\,\sqrt{\sum_{m=1}^M \frac{\norm{p^k- p^{(m)}}^2}{S^{(m)}}}}\right)$ otherwise. This rule for setting up $t^k$ is a mere artifice to model \cref{eq:Unbalanced-prox}. 
Indeed, $\mathtt{dist}_\L(\theta^k) = \norm{ \mathtt{Proj}_\L(\theta^k)-\theta^k}$ reduces to $\sqrt{\sum_{m=1}^M \frac{\norm{p^k- p^{(m)}}^2}{S^{(m)}}}$ thanks to \cref{prop_2}.

\paragraph{Convergence analysis}
The convergence analysis of \cref{alg} can be summarized as follows.
\begin{theorem}[MAM's convergence analysis]
\begin{itemize}
    \item [a)] (Deterministic MAM.) Consider \cref{alg} with the choice $\mathcal{M}^k=\col{1,\ldots,m}$ for all $k$. Then the sequence of points $\{p^k\}$ generated by the algorithm converges to a point $\bar p$. If the measures are balanced, 
    then $\bar p$ is a balanced WB; otherwise, $\bar p$ is a $\gamma$-unbalanced WB.
\item [b)] (Randomized MAM.) Consider \cref{alg} with the choice $\mathcal{M}^k\subset \col{1,\ldots,m}$
as in \cref{rem:random}.
Then the sequence of points $\{p^k\}$ generated by the algorithm converges almost surely to a point $\bar p$. If the measures are balanced, 
    then $\bar p$ is almost surely a balanced WB; otherwise, $\bar p$ is almost surely a $\gamma$-unbalanced WB.
\end{itemize}
\end{theorem}
\begin{proof}
    It suffices to show that \cref{alg} is an implementation of the (randomized) DR algorithm and invoke \cref{theo:DR} for item a) and \cref{theo:DR-random} for item b).    
    To this end, we first rely on \cref{prop_2} 
    to get that the projection of $\theta^k$ onto the balanced subspace $\L$ is given by $\theta^{(m),k}_{rs} +\frac{(p^k_r-p^{(m)}_r)}{S^{(m)}}$, $s=1,\ldots,S^{(m)}$,\;$r=1,\ldots,R$,\; $m=1,\ldots,M$,
where $p^k$ is computed at Step 1 of the algorithm, and the marginals $p^{(m)}$ of $\theta^k$ are  computed at Step 0 if $k=0$ or at Step 3 otherwise. Therefore, $\mathtt{dist}_\L(\theta^k) = \norm{ \mathtt{Proj}_\L(\theta^k)-\theta^k}=\sqrt{\sum_{m=1}^M \frac{\norm{p^k- p^{(m)}}^2}{S^{(m)}}}$. Now, given the rule for updating $t^k$ in \cref{alg} we can define the auxiliary variable $\pi^{k+1}$ as $\pi^{k+1} = 
\theta^k +t^k(\mathtt{Proj}_{\L}(\theta^k)-\theta^k)$, or alternatively,
\begin{equation}\label{eq:aux:theo}
\pi^{(m),k+1}_{rs} = 
\theta^{(m),k}_{rs} +t^k\frac{(p^k_r-p^{(m)}_r)}{S^{(m)}},\quad s=1,\ldots,S^{(m)},\;r=1,\ldots,R,\; m=1,\ldots,M.
\end{equation}
In the balanced case, $t^k=1$ for all $k$  (because $\gamma=\infty$) and thus $\pi^{k+1}$ is as in \cref{row_sol}. Otherwise, $\pi^{k+1}$ is as in \cref{eq:Unbalanced-prox} (see the comments after \cref{alg}). In both cases, $\pi^{k+1}$ coincides
with the auxiliary variable at the first step of the DR scheme \cref{DR} (see the developments at the beginning of this section).
Next, observe that to perform the second step of \cref{DR} we need to assess $y=2\pi^{k+1}-\theta^k$, which is
thanks to the above formula for $\pi^{k+1}$ given by
$y^{(m)}_{rs} = \theta^{(m),k}_{rs} +2\,t^k\,\frac{p_r^k-p_r^{(m)}}{S^{(m)}}$, $s=1,\ldots,S^{(m)}$,\;$r=1,\ldots,R$,\; $m=1,\ldots,M$.

As a result, for the choice $\mathcal{M}^k=\col{1,\ldots,M}$ for all $k$, Step 2 of \cref{alg} yields, thanks to \cref{prop_3}, $\hat \pi^{k+1}$ as at 
the second step of \cref{DR}. Furthermore, the updating of $\theta^{k+1}$ in the latter coincides with the rule in \cref{alg}: for $s=1,\ldots,S^{(m)}$, $r=1,\ldots,R$, and $m=1,\ldots,M$,
\begin{align*}
    \theta^{(m),k+1}_{rs} &= \theta^{(m),k}_{rs}  +\hat \pi^{(m),k+1}_{rs}-\pi^{(m),k+1}_{rs} = \theta^{(m),k}_{rs}  +\hat \pi^{(m),k+1}_{rs}-\left(\theta^{(m),k}_{rs} +t^k\frac{(p^k_r-p^{(m)}_r)}{S^{(m)}}\right)\\
   &=\hat \pi^{(m),k+1}_{rs}-t^k\frac{(p^k_r-p^{(m)}_r)}{S^{(m)}}.
\end{align*}
Hence, for the choice $\mathcal{M}^k=\col{1,\ldots,M}$ for all $k$, \cref{alg} is the DR Algorithm \cref{DR} applied to the WB \cref{pbm}. 
\Cref{theo:DR} thus ensures that the sequence $\{\pi^k\}$ as defined above converges to some $\bar \pi$ solving \cref{pbm}. To show that $\{p^k\}$ converges to a barycenter, let us first use the property that $\L$ is a linear subspace to obtain the decomposition $\theta=\mathtt{Proj}_\L(\theta)+ \mathtt{Proj}_{\L^\perp}(\theta)$ that allows us to rewrite the auxiliary variable $\pi^{k+1}$ differently:
\[
\pi^{k+1} = 
\theta^k +t^k(\mathtt{Proj}_{\L}(\theta^k)-\theta^k)
=\theta^k -t^k\mathtt{Proj}_{\L^\perp}(\theta^k).
\]
Let us denote $\tilde \pi^{k+1}:=\mathtt{Proj}_\L(\pi^{k+1})$. 
Then $\tilde \pi^{k+1}= 
\mathtt{Proj}_\L(\theta^k -t^k\mathtt{Proj}_{\L^\perp}(\theta^k))=\mathtt{Proj}_\L(\theta^k )$,
and thus \cref{prop_2} yields
\[
\tilde \pi^{(m),k+1}_{rs} = \theta^{(m),k}_{rs} +\frac{p_r^k-p_r^{(m)}}{S^{(m)}} \quad s=1,\ldots,S^{(m)},\;r=1,\ldots,R,\; m=1,\ldots,M,
\]
which in turn gives (by recalling that $\sum_{s=1}^{S^{(m)}}\theta^{(m),k}_{rs}=p^{(m)}_r$): 
$\sum_{s=1}^{S^{(m)}}\tilde \pi^{(m),k+1}_{rs}=p_r^k$,  $r=1,\ldots,R$,\; $m=1,\ldots,M$. As $\lim_{k\to \infty}\pi^k =\bar \pi$, 
$\lim_{k\to \infty}\tilde \pi^k = \lim_{k\to \infty} \mathtt{Proj}_\L (\pi^k) = \mathtt{Proj}_\L (\bar \pi) =: \tilde \pi$.
Therefore, for all $r=1,\ldots,R$, $m=1,\ldots,M$, the following limits are well defined:
\begin{equation}
    \bar p_r:= \sum_{s=1}^{S^{(m)}}\tilde \pi^{(m)}_{rs}=
\lim_{k\to \infty }\sum_{s=1}^{S^{(m)}}\tilde \pi^{(m),k+1}_{rs}=\lim_{k\to \infty } p_r^k.
\end{equation}

We have shown that the whole sequence $\{p^k\}$ converges to $\bar p$. By recalling that $\bar \pi$ solves \cref{pbm}, we conclude that in the balanced setting $\tilde \pi = \bar \pi$ and thus $\bar p$ is a WB. On the other hand, in the unbalanced setting, $\bar p$ above is  a $\gamma$-unbalanced WB according to \cref{def:UWB}.

The proof of item b) is a verbatim copy of the above proof: the sole difference, given the assumptions on the choice of $\mathcal{M}^k$, is that we need to rely on \cref{theo:DR-random} (and not on \cref{theo:DR} as previously done) to conclude that $\{\pi^k\}$ converges almost surely to some $\bar \pi$ solving \cref{pbm}. Thanks to the continuity of the orthogonal projection onto the subspace $\L$, the limits above yield almost surely convergence of $\{p^k\}$ to a barycenter $\bar p$.
\end{proof}

\if{
\subsection{Algorithm's Extension: the Full Unbalanced Setting}\label{sec:ext}
In our definition of unbalanced Wasserstein barycenter, we have relaxed only the marginals with respect to the barycenter $p$. In constrast, the approaches in \cite{Heinemann_Klatt_Munk_2022} and \cite{Sejourne_Peyre_Vialard_2023} also relax
the marginals with respect to $q^{(m)}$, $m=1,\ldots,M$. As we now show, extending 
MAM to cover such setting is a simple exercise. 

Here, motivated by \cite{Heinemann_Klatt_Munk_2022}, we choose another parameter $\eta>0$, redefine the distance matrices,  
\[
d^{(m)}_{rs}:=\alpha_m[\mathtt{d}(\xi_r,\zeta^{(m)}_s)-\eta], \quad r=1,\ldots,R,\; s=1,\ldots,S^{(m)},\; m=1,\ldots,M,\]
and replace problem~\cref{UWB}
with 
\[
\left\{
\begin{array}{llllllllll}
\displaystyle \min_{\pi } & \displaystyle \sum_{r=1}^R \sum_{s=1}^{S^{(1)}} d^{(1)}_{rs}\pi^{(1)}_{rs}&+\cdots +& \displaystyle\sum_{r=1}^R
 \sum_{s=1}^{S^{(M)}}d^{(M)}_{rs}\pi^{(M)}_{rs} &+& \gamma\,\mathtt{dist}_\L(\pi)\\
 \mbox{ }\\
 \mbox{s.t.} & \sum_{r=1}^R \pi^{(1)}_{rs} &&&\hspace{-1.3cm}\leq q^{(1)}_s,&\;s=1,\ldots,S^{(1)} \\
 &&\ddots &&\hspace{-1.3cm}\vdots\\
 &&&\sum_{r=1}^R \pi^{(M)}_{rs} &\hspace{-1.3cm}\leq  q^{(M)}_s,&\;s=1,\ldots,S^{(M)}\\[1em]
  &\pi^{(1)}\geq 0&\cdots &\pi^{(M)}\geq 0.
\end{array}
\right.
\]
In comparison with~\cref{UWB}, this problem has inequality constraints instead of equalities and penalizes the mismatches $\sum_{s=1}^{S^{(m)}}(q^{(m)}_s - \sum_{r=1}^R \pi^{(m)}_{rs})\geq 0$ with the penalty $\alpha_m \eta>0$, $m=1,\ldots,M$.
As $\sum_{m=1}^M \alpha_m \eta \sum_{s=1}^{S^{(m)}}q^{(m)}_s$ is a constant, it does not need to appear in the objective function above.
By defining 
\begin{gather*}
\Pi^{(m)}_{\leq}:=\left\{\pi^{(m)}\geq 0:\,\sum_{r=1}^R \pi^{(m)}_{rs}\leq q^{(m)}_s,\; s=1,\ldots,S^{(m)}\right\},\; m=1,\ldots,M,
\\
f^{(m)}_{\leq} (\pi^{(m)}):= 
\displaystyle \sum_{r=1}^R \sum_{s=1}^{S^{(m)}} =\alpha_m[\mathtt{d}(\xi_r,\zeta^{(m)}_s)-\eta]\pi^{(m)}_{rs} +\ind_{\Pi^{(m)}_{\leq}}(\pi^{(m)}),\quad m=1,\ldots,M,\\
f(\pi):=\sum_{m=1}^M f^{(m)}_{\leq}(\pi^{(m)})\quad \mbox{and}\quad
g(x)=    \gamma\,{\tt dist}_\L(\pi) ,
\end{gather*}   
the above problem can be recast as in formulation~\cref{pbm} and thus our approach applies.
Evaluating the proximal mapping of function $g$ has already been discussed in \cref{sec:proj}  (see \cref{eq:Unbalanced-prox}). It remains to discuss how to evaluate the  proximal mapping of the new function $f$ above.
On this account, let us define  the following set, for $\tau\geq 0$ arbitrary:
\[
\Delta_{\leq}(\tau):=\col{u \in \Re^R_+:\, \sum_{i=1}^R u_i\leq \tau}.
\]
\begin{proposition} 
Let the distance matrix $d$, function $f(\cdot)$ and set $\Delta_{\leq}(\cdot)$ be as above.
The proximal mapping 
\[\hat \pi:=\min_{\pi }\; f(\pi)+\frac{\rho}{2} \norm{\pi -y}^2 \]
can be computed exactly, in parallel along the columns of each transport plan $y^{(m)}$, as follows: 
for  all $m \in \{1,\ldots,M\}$,
\[
    \begin{pmatrix}\hat \pi^{(m)}_{1s}\\ \vdots \\ \hat \pi^{(m)}_{Rs})\end{pmatrix}=\,{\tt Proj}_{\Delta_{\leq}(q_s^{(m)})}
    \begin{pmatrix}
     y_{1s} - \frac{1}{\rho}d^{(m)}_{1s}\\
     \vdots  \\ 
     y_{Rs} - \frac{1}{\rho}d^{(m)}_{Rs} \end{pmatrix},
    \quad s=1,\ldots, S^{(m)}.
\]
\end{proposition}
\begin{proof}
    It suffices to proceed as in the proof of \cref{prop_3} by replacing equalities by inequalities.
\end{proof}
The projection onto $\Delta_{\leq }(\tau)$ is trivial if $\tau=0$. Let $\tau> 0$ and observe that projecting a vector $w \in \Re^R$ onto $\Delta_{\leq }(\tau)$ can be done in two steps:
\begin{itemize}
    \item First, define the non-negative vector $\check w$, with $ \check w_i:= \max\col{w_i,0}$, $i=1,\ldots,R$. 
    \item If $\sum_{i=1}^R \check w_i \leq \tau$, then $\mathtt{Proj}_{\Delta_{\leq}(\tau)}(w) = \check w$. Else, $\mathtt{Proj}_{\Delta_{\leq}(\tau)}(w) = \mathtt{Proj}_{\Delta_{R}(\tau)}(w)= \tau\,\mathtt{Proj}_{\Delta_{R}}(w/\tau)$,
    where the latter is the projection onto the $R+1$ simplex.
\end{itemize}
As a result, \cref{alg} can also be applied to solve the above more general unbalanced Wasserstein barycenter problem. All that is needed is to choose another penalty parameter $\eta>0$, modify the distance matrix (as above),
and replace $\Delta_R(q^{(m)}_s)$ with $\Delta_{\leq}(q^{(m)}_s)$ at line~\ref{line:proj} of \cref{alg}.

\begin{remark}
    Considering a forth case where $f(\cdot)$ is as above but $g(\cdot)=\ind_\L(\cdot)$ may give a poor definition of unbalanced WB. The reason is that the total mass of a solution $\bar \pi$ of \cref{pbm} is less than or equal to the minimal total masses of $q^{(m)}$, $m=1,\ldots,M$. Indeed, observe that in this case the multi-plans $\bar \pi^{(m)}$ share the same marginal (barycenter) $\bar p$, i.e., $\bar p_r = \sum_{s=1}^{S^{(m)}}\bar \pi^{(m)}_{rs}$, $r=1,\ldots,R$. Therefore,
    \[
    \sum_{r=1}^R  \bar p_r = \sum_{r=1}^R  \sum_{s=1}^{S^{(m)}}\bar \pi^{(m)}_{rs}
    = \sum_{s=1}^{S^{(m)}} \sum_{r=1}^R \bar \pi^{(m)}_{rs}\leq \sum_{s=1}^{S^{(m)}} q^{S^{(m)}}_s.
    \]
\end{remark}

}\fi

\if{
\paragraph{Constrained barycenter problem}
\paragraph{Relaxing additional marginal constraints}
\paragraph{Optimization of the support}
\label{sec:optim_support}

A more general formulation of MAM can be naturally derived following the theory of \cref{eq:free-WB} that updates both probabilities (say \textit{weights} in the unbalanced case) and support of the barycenter. This method is treated here as a new prospect that remains to be study more in depth since the present paper mainly focuses on the classic MAM (fixed support formulation) performance and properties.

\begin{algorithm}[htb]
\caption{\sc Wasserstein Barycenter }
  \label{alg2}
\begin{algorithmic}[1]

\State Set $k:=0$  and $\xi^0 := \xi $
\For{$k=0,1,2\ldots,$}

\State Compute $d^{(m)}_{rs}= \frac{1}{2}\|\xi_r^k-\zeta^{(m)}_s\|^2$, $m=1,\ldots,M$
\State Update the probabilities : apply \cref{alg} to the fixed support problem
\[
p^k \in\arg\min_{p \in \Delta_R}\;\sum_{m=1}^M \alpha_m{\tt OT}(p,q^{(m)};{\xi^k,\zeta^{(m)}})
\]

\State Update the support 
\begin{equation}\label{xi-pbm}
\xi^{k+1} =\arg\min_{\xi }\;\sum_{m=1}^M \alpha_m{\tt OT}(p^k,q^{(m)};{\xi,\zeta^{(m)}})
\end{equation}

\EndFor
  \end{algorithmic}
\end{algorithm}

\noindent Where ~\cref{xi-pbm} is explicitly given in \cref{eq:free-WB}.


}\fi
\section{Numerical Experiments}\label{sec:num}


This section illustrates the MAM's practical performance on some well-known datasets. 
The impact of different data structures is studied before the algorithm is compared to state-of-the-art methods. This section closes with an illustrative example of MAM to compute UWBs.  Numerical experiments were conducted using 20 cores (\textit{Intel(R) Xeon(R) Gold 5120 CPU}) and \textit{Python 3.9}.
The test problems and solvers’ codes are available for download in the link \url{https://ifpen-gitlab.appcollaboratif.fr/detocs/mam_wb}.

\subsection{Study on data structure influence} \label{parametric_section}
We start by evaluating the impact of conditions that influence the storage complexity and the algorithm performance. The main conditions are the \textit{sparsity} of the data and the \textit{number of distributions} $M$. 
Naturally, the denser the distributions or the more distributions are treated, the greater the storage. In these configurations, the time per iteration grows because the number of projects onto the simplex increases. To assess the impact of data sparsity and the number of measures on the algorithm's performance, we consider a fixed-support approach and experiment on datasets inspired by  \cite{IBP,Cuturi_Doucet_14}. The number of nested ellipses controls the density of a dataset: as exemplified in \cref{nested_ellipses}(a) and \cref{table1}, measures with only a single ellipse are very sparse. In contrast, a dataset with 5 nested ellipses is denser. 
%
\begin{figure}[htb] \label{nested_ellipses}
    \centering  \subfigure[A sample of datasets]
    {\includegraphics[width=0.45\textwidth]{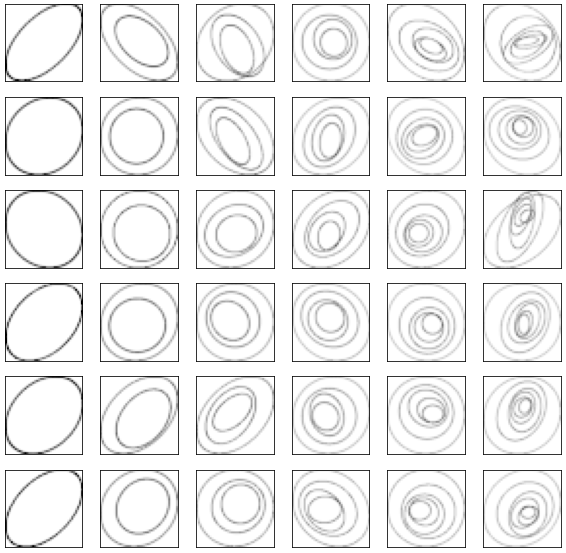}}
 \centering   \subfigure[MAM's number of iterations per second]{
     \includegraphics[width=0.5\textwidth]{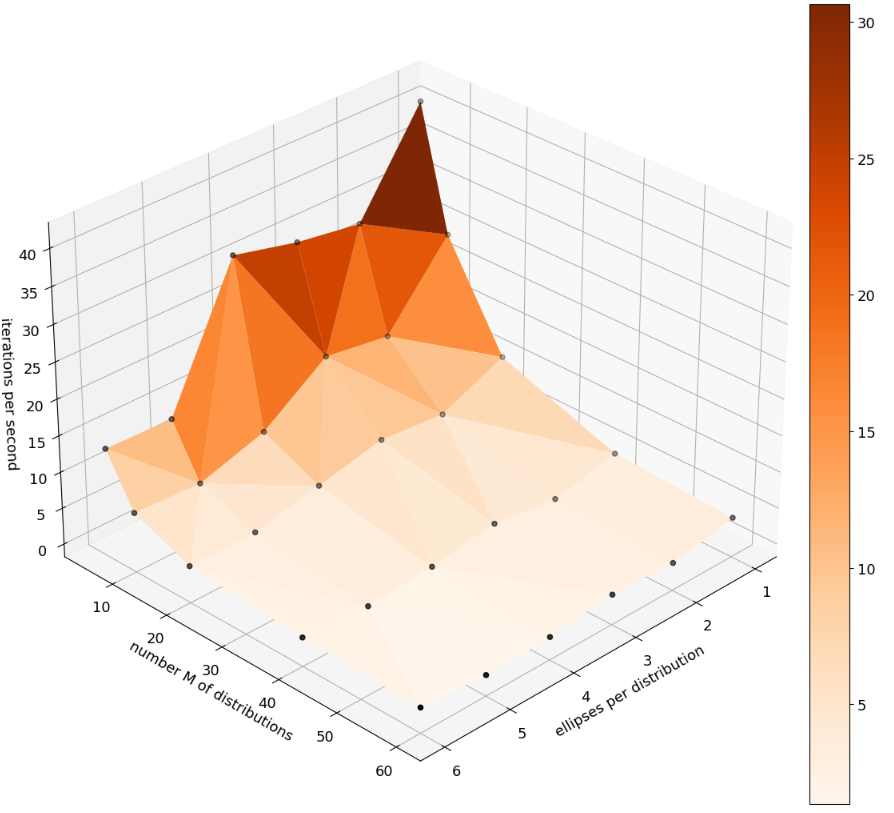}}
    \caption{(a) Sample of the artificial nested ellipses datasets. The first column is taken from the first dataset with 1 ellipse, the second column from the second dataset with 2 nested ellipses, and the sixth column with 6 nested ellipses.
    (b) Evolution of the number of iterations per second depending on the density or the number of distributions.}
\end{figure}
\begin{table}[h] \label{table1}
  \centering
  \caption{Mean density with the number of nested ellipses. The density has been calculated by averaging the ratio of non-null pixels per image over 100 generated pictures for each dataset sharing the same number of nested ellipses.}
  \renewcommand{\arraystretch}{1.5} 
  \begin{tabular}{|c|c|c|c|c|c|c|}
    \hline
    \rev{\textnormal{Number of ellipses}} & 1 & 2 & 3 & 4 & 5 & 6 \\
    \hline
    \textnormal{Density} ($\%$) & 29.0 & 51.4 & 64.3 & 70.9 & 73.5 & 75.0 \\
    \hline
  \end{tabular}
\end{table}
\rev{In this first experiment, we apply MAM with $\rho=100$ (without proper tuning) for every dataset. Only a single processor was considered to avoid CPU communication management. Figure~\ref{nested_ellipses}(b)} shows that, as expected, the execution time of an iteration increases with increasing density and number of measures.
%
The number of measures influences the method's speed more than density (this phenomenon can be due to the $numpy$ matrix management). 
This means the quantity of information in each measure does not seem to make the algorithm less efficient in terms of speed. Such a result is to be put in regard with algorithms such as B-ADMM \cite{J.Ye} that are particularly shaped for sparse datasets but less efficient for denser ones. \Cref{infl_support}  develops this further.
 \rev{Additionally, it is worth noting that the proposed method can harness parallel computation, enabling the distribution of work across the $M$ measures. This approach effectively mitigates the impact of the measure count on computational efficiency.}

\rev{
The growing dimensions of images have an impact on the computation time, as seen in \Cref{sec:MAM}. For example, when treating dense $K \times K$ images for a fixed support problem, the number of operations per probability density for the projection onto $\L$ is $O_1^\L =3\cdot K^2\cdot K^2=3\cdot K^4$ and onto the simplex $O_1^\Delta=K^2\cdot K^2\cdot\log(K^2)=2 K^4\cdot\log(K)$.
\begin{itemize}
    \item For a fixed-support problem with dense $(nK)\times(nK)$ images, $O_{nK}^\L =3\cdot (nK)^4=n^4\cdot O_1^\L$ and $O_{nK}^\Delta = n^4K^4\log(n^2K^2)\approx n^4 \cdot O_1^\Delta$.
    \item  For a fixed-support problem with dense $K\times\dots\times K=K^d$ measures, $O_{K^d}^\L =3\cdot (K^d)^2=K^{2d-4} \cdot O_1^\L $ and $O_{K^d}^\Delta=(K^d)^2\log(K^{d})= \frac{d}{2}K^{2d-4} \cdot O_1^\Delta$.
    \item For a free-support problem, in dimension $d$, with dense $K^d$ grids, the size of the support $R$ depends on the number $M$ of treated measures, $R=((K-1)M+1)^d$. Following the details of \Cref{sec:MAM}, $O_{free,K^d}^\L=3((K-1)M+1)^d K^d\approx M^d K^{2d-4} O_1^\L $ and $O_{free,K^d}^\Delta=((K-1)M+1)^d K^d\log(((K-1)M+1)^d)\approx \frac{d}{2} M^{d}K^{2d-4} \cdot O_1^\Delta$.
\end{itemize}
For instance, for a fixed-support problem with $40\times40$ images (see \cref{nested_ellipses}), the algorithm computes the projections for one measure in an average time of 0.01 seconds. However, for the free-support problem formulation with this dataset of 6 images, it takes 6 seconds per measure. Similarly, for a fixed-support problem with $40\times40\times40$ objects (ellipsoids with similar properties as in \cref{nested_ellipses} in 3D), the projections for one measure take 16 seconds.
}

\rev{
\subsection{Fixed-support approach} This section focuses on the fixed-support approach: $R$ in~\eqref{HugeLP} is equal to $K^2$, the number of pixels of a $K \times K$ image.
}
\subsubsection{Comparison with IBP}
The Iterative Bregman Projection (IBP) \cite{IBP} is a well-known algorithm for computing Wasserstein barycenters. As mentioned in the Introduction, IBP employs a regularizing function parameterized by $\lambda>0$, which impacts precision and must kept at a moderate magnitude to avoid numerical errors (double-precision overflow).
The experiment below sheds light on the differences between MAM and IBP and their advantages depending on the use. Our IBP code is inspired by the original {\tt MATLAB} code by G. Peyr\'e\footnote{\url{https://github.com/gpeyre/2014-SISC-BregmanOT}}.

\subsubsection{Qualitative comparison} \label{qualitative_sec}
Here, we use 100 images per digit of the MNIST database \cite{MNIST_database}, where each digit has been randomly translated and rotated. Each image has 40 $\times$ 40 pixels and can be treated as probability distributions after normalization.
\Cref{barycenter_evolution} displays intermediate solutions for digits $3, 4, 5$ at different time steps both for MAM and IBP. For the two methods, the hyperparameters have been tuned: 
for instance, $\lambda = 1700$ is the greatest lambda that enables IBP to compute the barycenter of the 3's dataset without double-precision overflow error. Regarding MAM, a range of values for $\rho>0$ have been tested for 100 seconds of execution, to identify which one provides good performance (for example, $\rho=50$ for the dataset  of $3$'s).
\begin{figure}[htb] \label{barycenter_evolution}
    \centering
    \includegraphics[width=.9\textwidth]{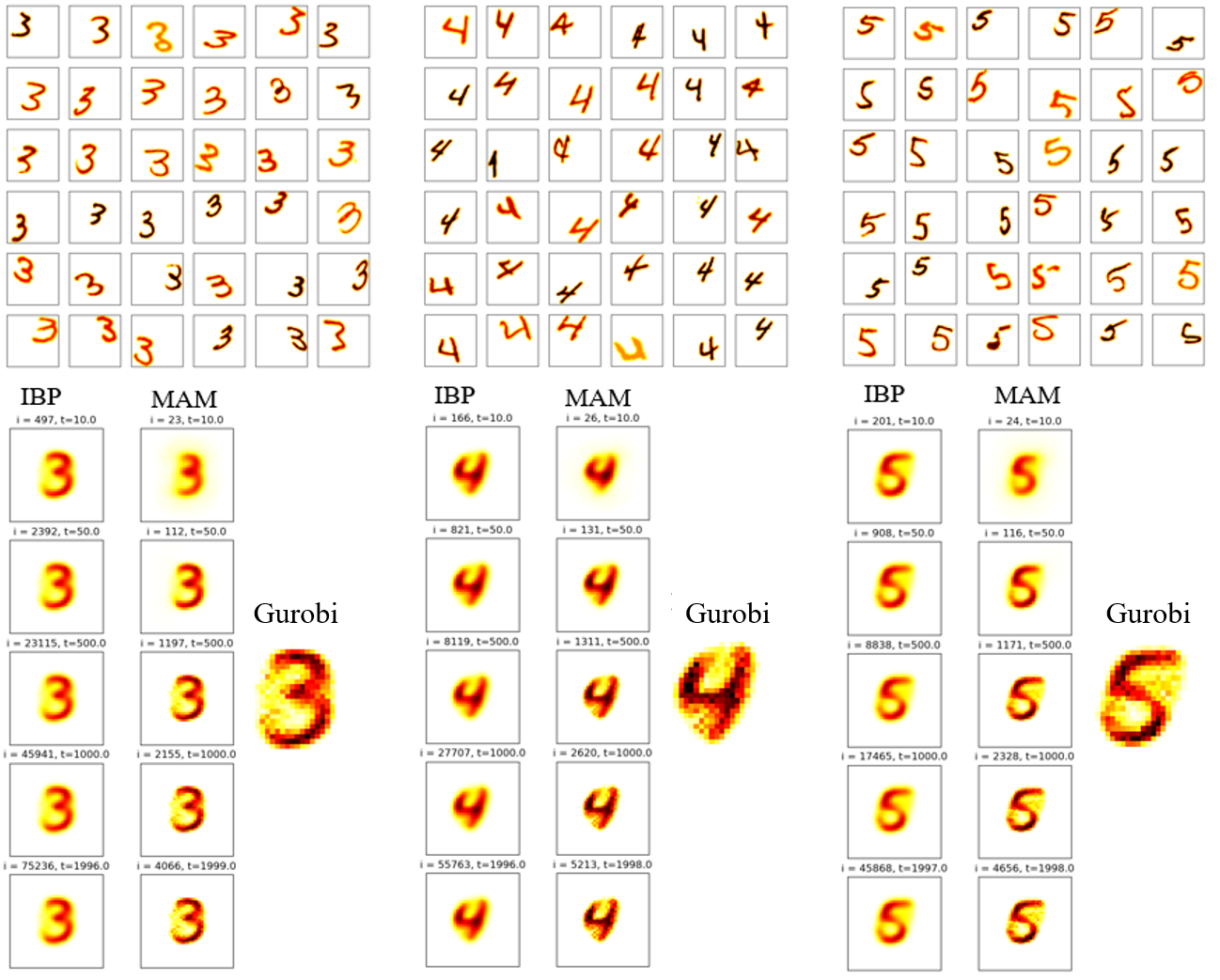}
    \caption{(top) For each digit 36 out of the 100 scaled, translated, and rotated images considered for each barycenter. (bottom) Barycenters after $t = 10, 50, 500, 1000, 2000$ seconds, where the left-hand-side is the IBP evolution of its barycenter approximation, the middle panel is MAM's evolution using 10 CPU
    and the right-hand-side is a solution computed by applying \textit{Gurobi} to the LP \cref{HugeLP}.}
\end{figure}
\Cref{barycenter_evolution} shows that, for each dataset, IBP gets quickly to a stable barycenter approximation. Such a point is obtained shortly after with MAM (less than 10 seconds after). However, MAM continues to move towards a sharper solution.
It is clear that the more CPUs used for MAM, the better. Furthermore, while IBP is not well-suitable for CPU parallelization  \cite{IBP_github, IBP, J.Ye}, MAM offers a clear advantage depending on the hardware at stake.

\subsubsection{Quantitative comparison} \label{quantitative_section}
Next, we benchmark MAM, randomized MAM and IBP on a dataset with 60 images per digit of the MNIST database \cite{MNIST_database}, where every digit is a normalized image 40 $\times$ 40 pixels. 
First, all three methods have their hyperparameters tuned thanks to a sensitivity study as explained in \Cref{qualitative_sec}. Then, at every time step an approximation of the computed barycenter is stored to compute the error $\bar W_2^2(p^{k}) -\bar W_2^2(p_{\rev{G}}) := \sum_{m=1}^M \frac{1}{M} W_2^2(\mu^k,\nu^{(m)}) - \sum_{m=1}^M \frac{1}{M}W_2^2(\mu_{\rev{G}},\nu^{(m)}))$, \rev{where $\mu_{G}$ is a fixed-support barycenter computed using \textit{Gurobi} to solve the LP \cref{HugeLP}}. 
\begin{figure}[h!] \label{IBP_vs_MAM}
    \centering
    \includegraphics[width=1\textwidth]{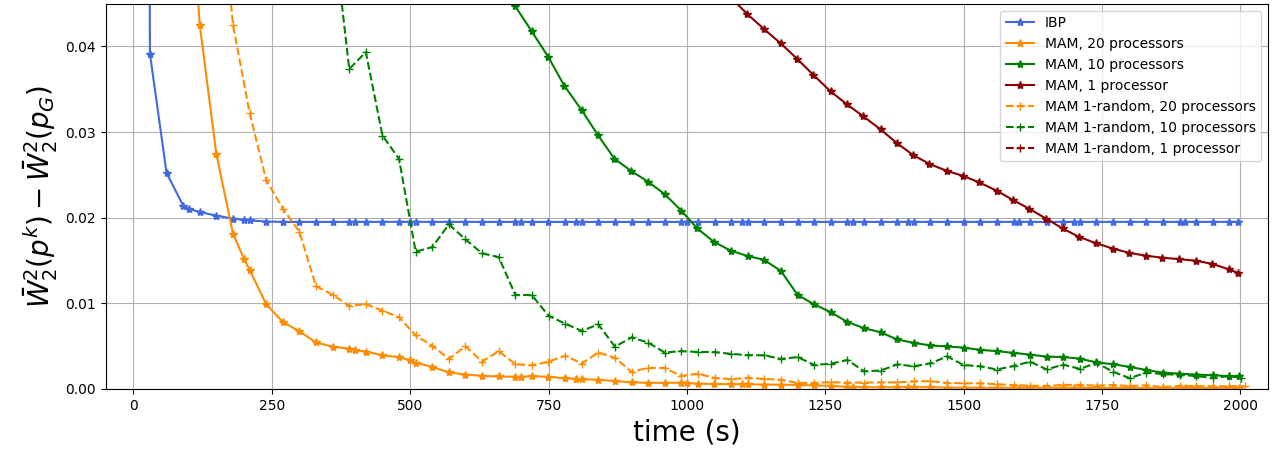}
    \caption{Evolution with respect to time of the difference between the Wasserstein barycenter distance of an approximation, $\bar W_2^2(p^{k})$, and the Wasserstein barycentric distance of the exact solution $\bar W_2^2(p_{G})$ given by the LP. The time step between two points is 30 seconds.}
\end{figure}

For this dataset, \cref{IBP_vs_MAM} shows that IBP is almost 10 times faster per iteration.
However, IBP computes a solution to the regularized model, not to the \rev{(fixed-support) WB linear problem~\eqref{HugeLP}. Instead, MAM does converge to a solution of~\eqref{HugeLP}.} So there is a threshold where the accuracy of MAM exceeds that of IBP: in our case, around 200s - for the computation with the greatest number of processors (see \cref{IBP_vs_MAM}).
Such a threshold always exists depending on the computational means (hardware). 
This quantitative study explains what has been exemplified with the images of \Cref{qualitative_sec}: the accuracy of IBP is bounded by the choice of $\lambda$, itself bounded by an overflow error. In contrast,  the MAM hyperparameter only impacts the convergence speed. For this dataset, the WB computed by IBP is within 2$\%$ of accuracy and thus reasonably good. However, as shown in Table 1 in \cite{J.Ye}, one can choose other datasets where IBP's accuracy might be unsatisfactory.

Furthermore, \cref{IBP_vs_MAM} exemplifies an attractive asset of randomized variants of MAM: in some configurations, randomized MAM is more efficient than (deterministic) MAM.
(The curve \textit{MAM 1-random, 1 processor} does not appear in the figure because it is above the y-axis value range due to its bad performance.)
Indeed, a trade-off exists between time spent per iteration and precision gained after an iteration. For example, with 10 processors, each processor treats six measures in the deterministic MAM, but only one is treated in the randomized MAM. Therefore, the time spent per iteration is roughly six times shorter in the latter, which counterbalances the loss of accuracy per iteration. On the other hand, when using 20 processors, only three measures are treated by each processor, and the trade-off is not worth it anymore: the gain in time does not compensate for the loss in accuracy per iteration. One should adapt the use of the algorithm with care since this trade-off conclusion is only heuristic and strongly depends on the underlying dataset and hardware. A sensitivity analysis is always a good thought for choosing the most effective amount of measures handled per processor while using the randomized MAM against the deterministic MAM

\subsubsection{Influence of the support} \label{infl_support}
This section echoes \Cref{parametric_section} and studies the influence of the support size. To do so, two datasets have been tested for MAM and IBP. The first dataset is already used in \Cref{quantitative_section}: 60 pictures of 3's taken from the  MNIST database \cite{MNIST_database}. The second dataset is also composed of these 60 images but each digit has been randomly translated and rotated in the same way as in \cref{barycenter_evolution}. Therefore, the union of the support of the second dataset is greater than the first one. 


\cref{fig_compare_support} presents two graphs that have been obtained just as in \Cref{quantitative_section}, but displaying the evolution in percentage: $\Delta W_{\%} := \frac{\bar W_2^2(p^{k}) -\bar W_2^2(p_{\rev{G}}) }{\bar W_2^2(p_{\rev{G}})} \times 100$. Once more, the hyperparameters have been fully tuned. The hyperparameter of the IBP method is smaller for the second dataset. \rev{Indeed, as stated in \cite{J.Ye}, the greater the support, the stronger the restrictions on $\lambda$, and thus, the less precise IBP.}
%
%
\begin{figure}[h] \label{fig_compare_support}
    \centering
    \includegraphics[width=.9\textwidth]{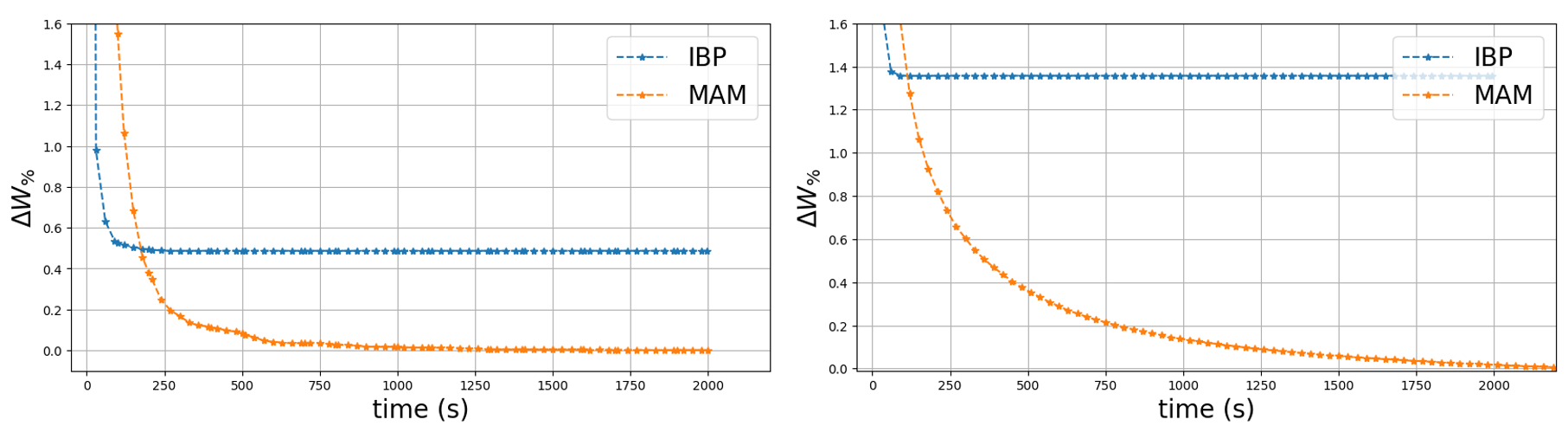}
    \caption{Evolution of the percentage of the distance between the exact solution of the barycenter problem and the computed solution using IBP and MAM method with 20 processors: (left) for the standard MNIST, (right) for the randomly translated and rotated MNIST. 
    }
\end{figure}


\subsubsection{Comparison with B-ADMM} \label{BADMM_vs_MAM}
This subsection compares MAM with the algorithm B-ADMM of \cite{B-ADMM} using the dataset and {\tt MATLAB} implementation provided by the authors at the link \url{https://github.com/bobye/d2_kmeans}. We omit  IBP in our analysis because it has already been shown in  \cite[Table I]{B-ADMM} that IBP is outperformed by B-ADMM in this dataset. As in \cite[Section IV]{B-ADMM}, we consider $M=1000$ discrete measures, each with a sparse finite support set obtained by clustering pixel colors of images.
The average number of support points is around $6$, and the barycenter's number of \rev{fixed-}support points is $R=60$.
\rev{The optimal value of~\eqref{HugeLP}} is $712.7$, computed in $10.6$ seconds by the Gurobi LP solver. We have coded MAM in {\tt MATLAB} to have a fair comparison with the {\tt MATLAB} B-ADMM algorithm provided at the above link. Since MAM and  B-ADMM use different stopping tests, we have set their stopping tolerances equal to zero and let the solvers stop with a maximum number of iterations. \Cref{tab:MAMxBADMM} below reports CPU time in seconds and the objective values yielded by the (approximated) barycenter $\tilde p$ computed by both solvers: $\bar W_2^2(\tilde p)$.
\begin{table}[htb]
\centering
  \caption{MAM vs B-ADMM. 
  B-ADMM code is the one provided by its designers without changing parameters (except the stopping set to zero and the maximum number of iterations). Both algorithms use the same initial point. 
  The optimal value of the WB barycenter for this dataset is $712.7$, computed by Gurobi in $10.6$ seconds. }
  \label{tab:MAMxBADMM}
\begin{tabular}{|c|cc|cc|}
\hline
\multirow{2}{*}{Iterations} & \multicolumn{2}{c|}{Objective value} & \multicolumn{2}{c|}{Seconds}       \\ \cline{2-5} 
                            & \multicolumn{1}{c|}{B-ADMM}  & MAM   & \multicolumn{1}{c|}{B-ADMM} & MAM  \\ \hline
100                         & \multicolumn{1}{c|}{742.8}   & 716.7 & \multicolumn{1}{c|}{1.1}    & 1.1  \\ \hline
200                         & \multicolumn{1}{c|}{725.9}   & 714.1 & \multicolumn{1}{c|}{2.4}    & 2.2  \\ \hline
500                         & \multicolumn{1}{c|}{716.5}   & 713.3 & \multicolumn{1}{c|}{5.6}    & 5.4  \\ \hline
1000                        & \multicolumn{1}{c|}{714.1}   & 712.9 & \multicolumn{1}{c|}{11.8}   & 10.8 \\ \hline
1500                        & \multicolumn{1}{c|}{713.5}   & 712.8 & \multicolumn{1}{c|}{18.9}   & 16.2 \\ \hline
2000                        & \multicolumn{1}{c|}{713.3}   & 712.8 & \multicolumn{1}{c|}{25.1}   & 21.6 \\ \hline
2500                        & \multicolumn{1}{c|}{713.2}   & 712.8 & \multicolumn{1}{c|}{31.0}   & 27.1 \\ \hline
3000                        & \multicolumn{1}{c|}{713.1}   & 712.7 & \multicolumn{1}{c|}{39.8}   & 32.4 \\ \hline
\end{tabular}
\end{table}

The results show that, for the considered dataset, MAM and B-ADMM are comparable regarding CPU time, with MAM providing more precise results. B-ADMM currently lacks a convergence analysis, unlike MAM.

\rev{
\subsection{Free-support approach} \label{free_support_a}
This section considers the \emph{free-support} problem (see \Cref{sec:background}, \cref{Th_Borg}), where the measures are supported on the same discrete grid in $\Re^2$ ($d=2$) and $\alpha_m=\frac{1}{M}$ for all $m=1,\dots,M$. The dataset we use is the one from~\cite{Altschuler_Boix-Adsera_2020}, 
illustrated in \cref{fig_dataset_altsch}. In this case, $M=10$ 
measures, $S=K^2=60^2$ and $R= ((K-1)M+1)^d = 591^2=349281$. 
The resulting LP problem is too large to be solved by standard solvers. Therefore, we employed the dedicated solver of \cite{Altschuler_Boix-Adsera_2020}, available at the link \url{https://github.com/eboix/high_precision_barycenters}.
%
 %
\begin{figure}[h] \label{fig_dataset_altsch}
    \centering
    \includegraphics[width=.7\textwidth]{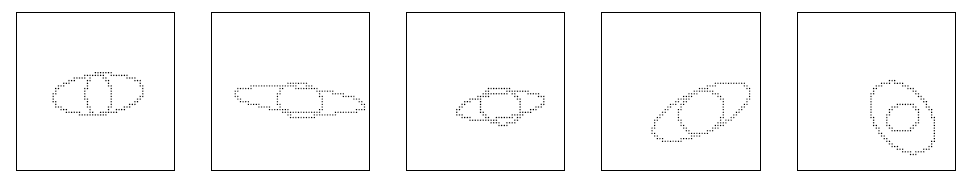}
    \caption{Five $60\times60$ images from the nested ellipses dataset in \cite{Altschuler_Boix-Adsera_2020}. } 
\end{figure}

\Cref{fig:Altschuler_vs_MAM} presents the evolution of the points computed by MAM along 9 hours of processing. The image on the right-hand side is an exact barycenter computed by the solver of \cite{Altschuler_Boix-Adsera_2020} after 3.5 hours.
We recall that \cite{Altschuler_Boix-Adsera_2020} handles the dual of~\eqref{HugeLP} by employing a geometry-based separation oracle. Once the dual is solved, the method recovers a primal vertex, yielding thus a sparse WB.
As a result, the right-hand side image in~\Cref{fig:Altschuler_vs_MAM} is sharp. Such an exact WB is sharper than the point provided by MAM after 9 hours.
%
%
\begin{figure}[h] \label{fig:Altschuler_vs_MAM}
    \centering
    \includegraphics[width=1\textwidth]{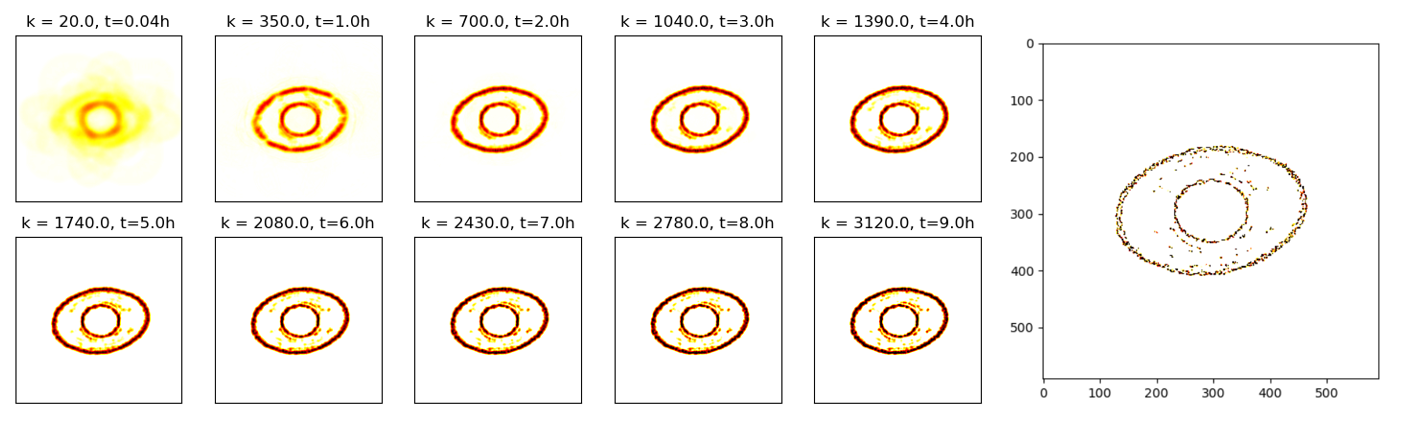}
    \caption{Evolution of the approximated MAM barycenter with time in regards with the exact barycenter of the Altschuler and Bois-Adsera algorithm computed in 4 hours \cite{Altschuler_Boix-Adsera_2020}. } 
\end{figure}
Despite the visual differences, the point provided by MAM is a Wasserstein barycenter. To see this, we compare the values of the objective function in~\eqref{HugeLP}, i.e., the 
Wassersein barycentric distance. 
%
The exact solution of the method in \cite{Altschuler_Boix-Adsera_2020} has a barycentric distance of $0.2666$.
After 1 hour of processing, our method had a barycenter distance of 0.2702, which improved to 0.2667 after 3.5 hours, when the solver \cite{Altschuler_Boix-Adsera_2020} halts. The slight visual difference stems from the fact that \cite{Altschuler_Boix-Adsera_2020} finds a vertex solution to WB problem while MAM does not. 

\Cref{fig:evol_pi_hat} illustrates MAM's iterative process. 
Let $\hat\pi^{(m),k}$ be the $m^{th}$ transportation plan computed by MAM at iteration $k$ (see Line~\ref{line:proj} of Algorithm~\ref{alg}). Note that $\hat W_2^2(p^k) := \sum_{m=1}^M\langle c^{(m)},\hat\pi^{(m),k} \rangle$  is an approximation to
$\bar W_2^2(p^k) := \sum_{m=1}^M W_2^2(\mu^k,\nu^{(m)}) $,  the exact function value (barycentric distance) at iteration $k$.
%
Furthermore, as it can be seen from the Douglas-Rachford algorithm in \cref{DR}, the transport plans $\hat\pi^{(m),k}$ do not necessarily respect the constraint embodied by $\L$ and is thus infeasible to the WB problem~\eqref{HugeLP}. 
Thus, the approximate value $\hat W_2^2(p^k) $ has to be seen in perspective with the distance of $\hat \pi^k$ to $\L$, i.e., ${\tt dist}_\L(\hat\pi^k)$. 
\Cref{fig:evol_pi_hat} shows the evolution of the 
approximate barycentric distance $\hat W_2^2(p^k) $, infeasibility measure ${\tt dist}_\L(\hat\pi^k)$, exact  barycentric distance $\bar W_2^2(p^k)$, and optimal value $\bar W_2^2(p_{exact}) = 0.2666$. 
%
%
\begin{figure}[h] \label{fig:evol_pi_hat}
    \centering
    \includegraphics[width=.8\textwidth]{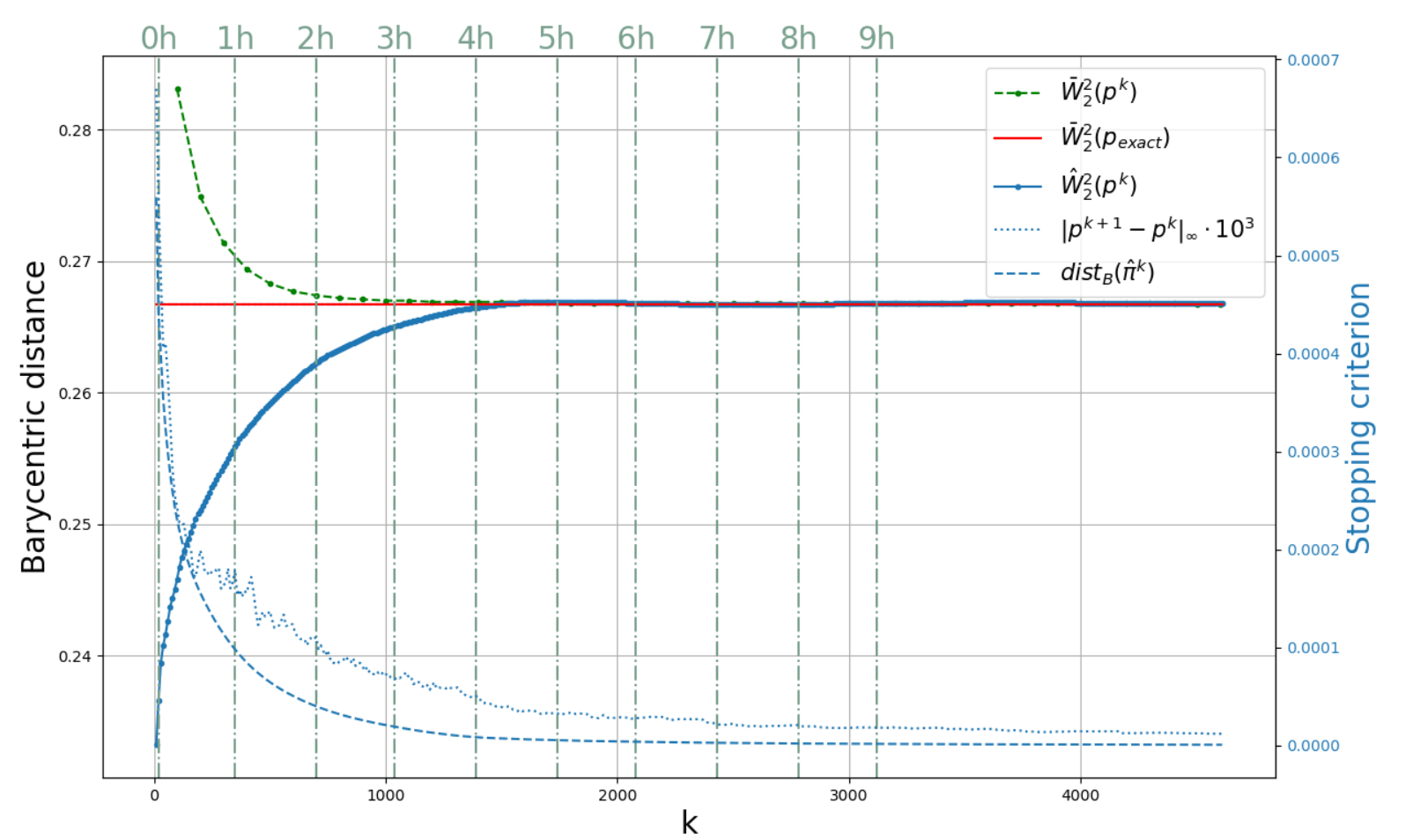}
    \caption{Evolution of the approximated Wasserstein barycenter distance $\hat W_2^2(p^k)$ with iterations (k) and time. } 
\end{figure}
We emphasize that $\bar W_2^2(p^k)$ is computed (by Gurobi) after terminating MAM, while $\hat W_2^2(p^k) $ and ${\tt dist}_\L(\hat\pi^k)$ are computed along the iterative process. After 3.5 hours, MAM provides $\bar W_2^2(p^k)=0.2667$,
$\hat W_2^2(p^k)=0.2658$ and ${\tt dist}_\L(\hat\pi^k) =1.27\cdot 10^{-5}$.

Because of its structure, the algorithm of \cite{Altschuler_Boix-Adsera_2020} cannot provide intermediary approximations of the barycenters, which is a disadvantage of the method over MAM. 
As an example, we consider $M=10$ images $40 \times 40$ presented in \Cref{quantitative_section}. Although smaller, these images are much denser than the ones in \Cref{fig_dataset_altsch} and thus the WB problem is more complicated. While MAM can provide \emph{free-support} WB approximations all along its iterative process, the solver of \cite{Altschuler_Boix-Adsera_2020} could not provide a solution after 50 hours of processing.



}

\if{ 
when the \emph{free-support} approach is taken, computing a WB amounts to solve~\eqref{HugeLP} by considering all points $\xi \in \Xi$, thus yielding an LP of astronomical size. 
\Cref{tab:MAM-random} compares MAM with three random variants in a serial environment with only one processor. To this end, we have fixed the CPU time for the random variants to the values of \cref{tab:MAMxBADMM}: the first column in \cref{tab:MAM-random} is the last one in \cref{tab:MAMxBADMM}, and the last column in \cref{tab:MAMxBADMM} is the third one in \cref{tab:MAMxBADMM}.
For instance, solver MAM-50 is the random variant of \cref{alg} with $50$ measures randomly chosen at every iteration.

\begin{table}[htb]
\centering
  \caption{Comparison of MAM (deterministic) and three randomized variants on the dataset of \cite{B-ADMM}. }
  \label{tab:MAM-random}
\begin{tabular}{|c|ccccl|}
\hline
\multirow{2}{*}{Seconds} & \multicolumn{5}{c|}{Objective value}                                                                                                                \\ \cline{2-6} 
                         & \multicolumn{1}{c|}{MAM-5} & \multicolumn{1}{c|}{MAM-50} & \multicolumn{1}{c|}{MAM-100} & \multicolumn{1}{c|}{MAM-500} & MAM                        \\ \hline
1.1                      & \multicolumn{1}{c|}{723.2} & \multicolumn{1}{c|}{716.9} & \multicolumn{1}{c|}{716.5}   & \multicolumn{1}{c|}{716.6}   & 716.7                      \\ \hline
2.2                      & \multicolumn{1}{c|}{716.4} & \multicolumn{1}{c|}{713.9}  & \multicolumn{1}{c|}{714.4}   & \multicolumn{1}{c|}{714.1}   & 714.1                      \\ \hline
5.4                      & \multicolumn{1}{c|}{714.0} & \multicolumn{1}{c|}{713.3}  & \multicolumn{1}{c|}{713.3}   & \multicolumn{1}{c|}{713.2}   & 713.3                      \\ \hline
10.8                     & \multicolumn{1}{c|}{713.3} & \multicolumn{1}{c|}{712.9}  & \multicolumn{1}{c|}{712.9}   & \multicolumn{1}{c|}{712.9}   & 712.9                      \\ \hline
16.2                     & \multicolumn{1}{c|}{713.0} & \multicolumn{1}{c|}{712.8}  & \multicolumn{1}{c|}{712.9}   & \multicolumn{1}{c|}{712.8}   & \multicolumn{1}{c|}{712.8} \\ \hline
21.2                     & \multicolumn{1}{c|}{713.0} & \multicolumn{1}{c|}{712.8}  & \multicolumn{1}{c|}{712.8}   & \multicolumn{1}{c|}{712.8}   & \multicolumn{1}{c|}{712.8} \\ \hline
27.1                     & \multicolumn{1}{c|}{712.9} & \multicolumn{1}{c|}{712.8}  & \multicolumn{1}{c|}{712.8}   & \multicolumn{1}{c|}{712.8}   & \multicolumn{1}{c|}{712.8} \\ \hline
32.4                     & \multicolumn{1}{c|}{712.8} & \multicolumn{1}{c|}{712.8}  & \multicolumn{1}{c|}{712.7}   & \multicolumn{1}{c|}{712.7}   & \multicolumn{1}{c|}{712.7} \\ \hline
\end{tabular}
\end{table}
In same cases, the randomized variant performs slightly better. For instance, the variants MAM-50 and MAM-100 manage to reach the optimal value faster than the deterministic variant: see the penultimate line in \cref{tab:MAM-random}. 
}\fi

\subsection{Unbalanced Wasserstein Barycenter} \label{sec:resul-UWB}

This section treats a particular example to illustrate the interest in using UWB. The artificial dataset is composed of 50 images with resolution $80\times80$. Each image is divided into four squares. The top left, bottom left, and bottom right squares are randomly filled with double nested ellipses and the top right square is always empty as exemplified in \cref{dataset_unbalanced}. In this example, every image is normalized to depict a probability measure so that we can compare \rev{(fixed-support)} WB and UWB.
\begin{figure}[h] \label{dataset_unbalanced}
    \centering
    \includegraphics[width=.7\textwidth]{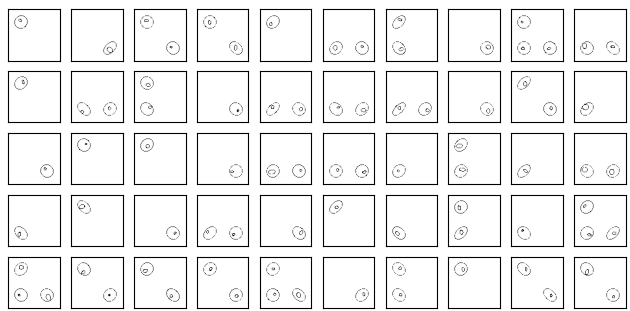}
    \caption{Dataset composed of 50 pictures with nested ellipses randomly positioned in the top left, bottom right, and left corners.}
\end{figure}

With respect to \cref{UWB}, one set of constraints is relaxed and the influence of the hyperparameter $\gamma$ is studied. If $\gamma$ is large enough (i.e. greater than $\norm{{\tt vec}(c)} \approx 1000$, see \cref{prop_gamma}), the problem boils down to the standard WB problem since the example deals with probability measures: the resulting UWB is indeed a WB.  When decreasing $\gamma$ the transportation costs take more importance than the distance to $\L$ which is more and more relaxed. Therefore, as illustrated by \cref{res_gamma}, the resulting UWB splits the image into four parts, giving visual meaning to the \rev{fixed-support} barycenter.

\begin{figure}[h] \label{res_gamma}
    \centering
    \includegraphics[width=1.\textwidth]{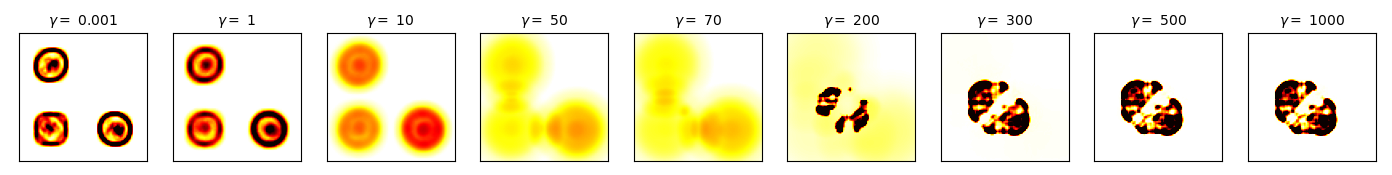}
    \caption{UWB computed with MAM for different values of $\gamma$.}
\end{figure}

In the same vein, \cref{UWB_ill_MAM} provides an illustrative application of MAM for computing UWB in another dataset.

\begin{figure}[h] \label{UWB_ill_MAM}
    \centering
    \includegraphics[width=.9\textwidth]{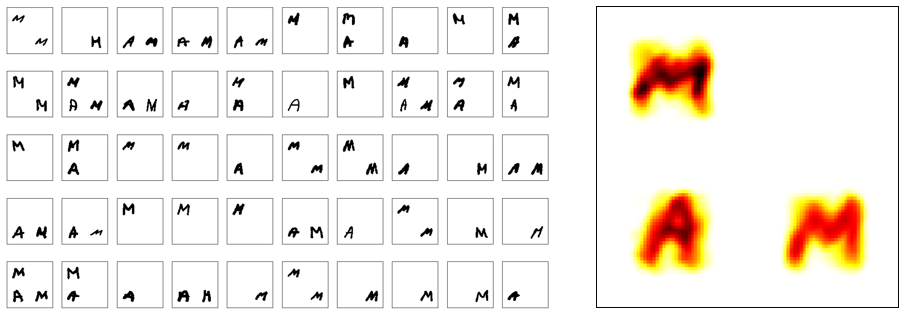}
    \caption{\textit{(left)} UWB for a dataset of letters M-A-M built in the same logic than \cref{dataset_unbalanced} with 50 figures: \textit{(right)} resulting UWB with $\gamma=0.01$, computed in 200 seconds using 10 processors.} 
\end{figure}

\bibliographystyle{plain}
\bibliography{article}

\newpage
\appendix

\end{document}